\documentclass[12pt]{article}
\usepackage{graphicx}
\usepackage{amssymb,amsmath,amstext}
\usepackage{amsthm}
\usepackage[colorinlistoftodos,bordercolor=orange,backgroundcolor=orange!20,linecolor=orange,textsize=scriptsize]{todonotes}
\usepackage[margin=1in]{geometry}
\usepackage{multirow}
\usepackage{caption}
\usepackage{subcaption}
\usepackage{booktabs,tabularx}
\usepackage{xspace}
\usepackage{placeins}
\usepackage[colorlinks=true, urlcolor=blue, linkcolor=blue, citecolor=blue]{hyperref}

\theoremstyle{plain}%
\newtheorem{theorem}{Theorem}
\numberwithin{theorem}{section}
\newtheorem{proposition}[theorem]{Proposition}
\newtheorem{lemma}[theorem]{Lemma}
\newtheorem{corollary}[theorem]{Corollary}
\newtheorem{definition}[theorem]{Definition}

\newtheorem{example}[theorem]{Example}
\newtheorem{question}[theorem]{Question}
\newtheorem{remark}[theorem]{Remark}

\newcommand{\Z}{\mathbb{Z}}
\newcommand{\PP}{\mathbb{P}}
\newcommand{\R}{\mathbb{R}}

\DeclareMathOperator\conv{conv}
\DeclareMathOperator\vol{vol}
\DeclareMathOperator\val{val}
\DeclareMathOperator\trop{trop}
\DeclareMathOperator\codim{codim}
\DeclareMathOperator\interior{int}
\DeclareMathOperator\vertices{vertices}

\newcommand\gfan{\texttt{Gfan}\xspace}

\newcommand\polymake{\texttt{polymake}\xspace}

\newcommand\topcom{\texttt{TOPCOM}\xspace}

\newcommand\arXiv[1]{\href{http://arxiv.org/abs/#1}{\tt arXiv:#1}}

\allowdisplaybreaks

\date{}
\begin{document}

\title{\bf  Moduli of Tropical Plane Curves}

\author{Sarah Brodsky,  Michael Joswig, 
Ralph Morrison and  Bernd Sturmfels}

\maketitle

\begin{abstract} \noindent We study the moduli space of metric graphs that arise from tropical plane
  curves.  There are far fewer such graphs than tropicalizations of 
  classical plane curves.  For fixed genus $g$, our moduli space is a stacky fan whose cones are indexed by regular
  unimodular triangulations of Newton polygons with $g$ interior lattice points.  It has dimension
  $2g+1$ unless $g \leq 3$ or $g = 7$.  We compute these spaces explicitly for $g \leq 5$.
\end{abstract}

\section{Introduction}

Tropical plane curves $C$ are dual to regular subdivisions of their Newton polygon $P$. The 
tropical curve $C$ is \emph{smooth}
if that subdivision is a unimodular triangulation $\Delta$, i.e.~it consists of triangles whose only lattice points are
its three vertices.  The \emph{genus} $g = g(C)$ is the number of interior lattice points of $P$. Each 
bounded edge of $C$ has a well-defined lattice length.
The curve $C$ contains a subdivision of a metric graph of genus $g$ with vertices of valency $\geq 3$ as in
\cite{BPR}, and this subdivision is unique for $g\ge 2$.  The underlying graph $G$ is planar and has $g$ distinguished
cycles, one for each interior lattice point of $P$.  We call $G$ the \emph{skeleton} of $C$.  It is the smallest subspace of $C$ to which $C$ admits a deformation retract.   

While the metric on $G$
depends on $C$, the graph is determined by $\Delta$.  For an illustration see Figure~\ref{figure:type3_case1intro}. The
triangulation $\Delta$ on the left defines a family of smooth tropical plane curves of degree four.  Such a curve has
genus $g=3$. Its skeleton $G$ is shown on the right.

\begin{figure}[h]
\centering
\includegraphics[scale=0.96]{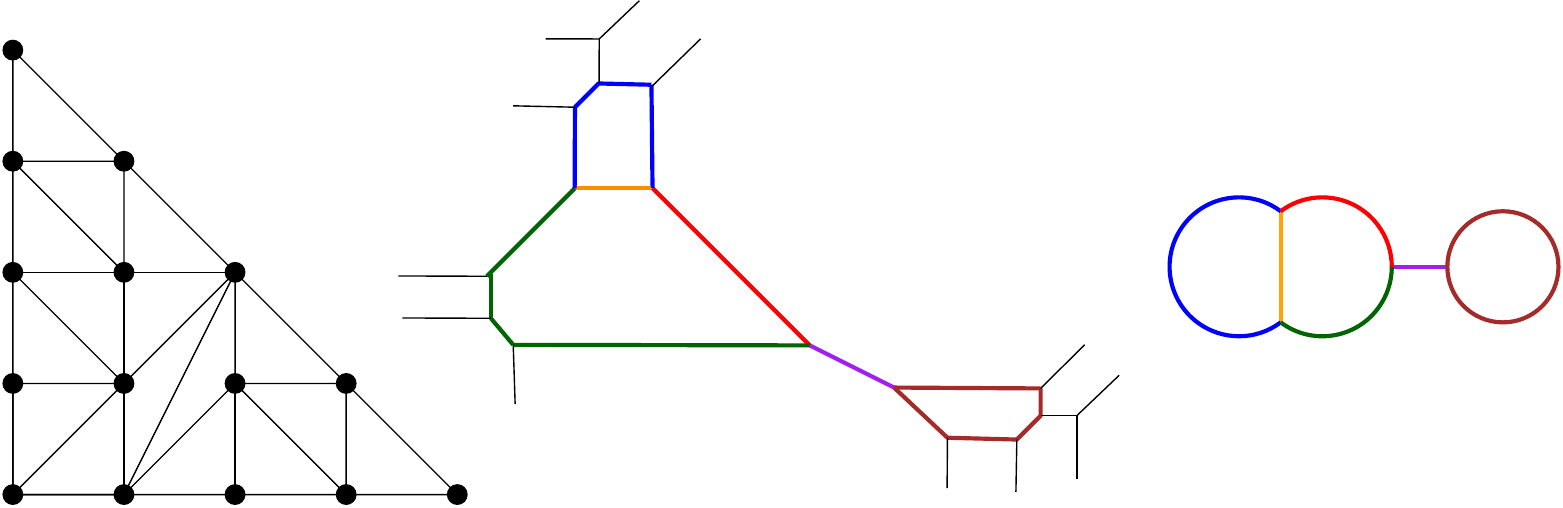}
\caption{Unimodular triangulation, tropical quartic, and  skeleton}
\label{figure:type3_case1intro}
\end{figure}

For basics on tropical geometry and further references the reader is referred to \cite{IMS, MS}.
 Let $\mathbb{M}_g$ denote the \emph{moduli space of metric graphs} of genus $g$.  
The moduli space $\mathbb{M}_g$
is obtained by gluing together finitely many orthants $\mathbb{R}_{\geq 0}^m$, $m \leq 3g{-}3$,
 one for each combinatorial type of graph, modulo the identifications corresponding to
graph automorphisms.  These automorphisms endow
the moduli space  $\mathbb{M}_g$ with the structure of a 
\emph{stacky fan}. 
We refer to \cite{BMV, Chan} for the definition of $\mathbb{M}_g$, 
combinatorial details, and applications in algebraic geometry.
 The maximal cones of $\mathbb{M}_g$
correspond to trivalent graphs of genus
$g$. These have $2g-2$ vertices and $3g-3$ edges, so $\mathbb{M}_g$ is pure of dimension $3g-3$.  
The number of trivalent graphs for $g = 2,3,\ldots,10\,$ is $ \,2, 5, 17, 71, 388, 2592, 21096, 204638,
2317172$; see \cite{Ba} and \cite[Prop.~2.1]{Chan}.

Fix a (convex) lattice polygon $P$ with $g = \# (\interior(P) \cap \Z^2)$.   Let
$\mathbb{M}_P$ be the closure in $\mathbb{M}_g$ of the set of metric graphs that are realized by
smooth tropical plane curves with Newton polygon $P$. For a fixed regular unimodular triangulation
$\Delta$ of $P$, let $\mathbb{M}_\Delta $ be the closure of the cone of metric graphs from
tropical curves dual to $\Delta$. These curves all have the same skeleton $G$, and
$\mathbb{M}_\Delta$ is a convex polyhedral cone in the orthant $\mathbb{R}_{\geq 0}^{3g-3}$ of metrics on $G$. Working modulo automorphisms of $G$, we identify $\mathbb{M}_\Delta$
with its image in the stacky fan $\mathbb{M}_g$.

  Now fix the skeleton $G$ but vary the triangulation.  The resulting subset of
  $\mathbb{R}_{\geq 0}^{3g-3}$ is a finite union of closed convex polyhedral cones,
  so it can be given the  structure of a polyhedral fan. Moreover, by appropriate subdivisions,
  we can choose  a fan structure that is invariant under the symmetries of $G$, and
  hence the image in  the moduli space $\mathbb{M}_g$ is a stacky fan:
\begin{equation}
\label{eq:MPGdecompose}
\mathbb{M}_{P,G} \quad := \ 
\bigcup_{\Delta \text{ triangulation of P} \atop \text{with skeleton } G} \mathbb{M}_\Delta.
\end{equation}

We note that $\mathbb{M}_P$ is represented inside $\mathbb{M}_g$ 
by finite unions of convex polyhedral cones:
\begin{equation}
\label{eq:MPdecompose}
 \mathbb{M}_P \quad = \bigcup_{G \text{ trivalent graph}\atop \text{of genus } g} \!\!\! M_{P,G}
 \quad = \ \ \bigcup_{\Delta \text{ regular unimodular} \atop \text{triangulation of } P} \!\!\!\! \mathbb{M}_\Delta.
\end{equation}
The \emph{moduli space of tropical plane curves of genus $g$} is the following stacky fan inside
$\mathbb{M}_g$:
\begin{equation}
\label{eq:Mgplanar}
\mathbb{M}_{g}^{\rm planar} \ := \ \, \bigcup_P \,\mathbb{M}_P  .
\end{equation}
Here $P$ runs over isomorphism classes 
of  lattice polygons with $g$ interior lattice points.  
The number of such classes is finite by  Proposition~\ref{prop:finite}. 

This paper presents a computational study of
 the moduli spaces $\mathbb{M}_{g}^{\rm planar}$. We construct
  the decompositions in (\ref{eq:MPdecompose}) and
(\ref{eq:Mgplanar}) explicitly.  Our first result reveals the dimensions:

\begin{theorem}
\label{thm:dimension}
For all $g \geq 2$ there exists a  lattice polygon $P$ with $g $ interior lattice points such
that $\mathbb{M}_P$ has the dimension expected from classical algebraic geometry, namely,
\begin{equation}
\label{eq:dimformula}
\dim(\mathbb{M}_g^{\rm planar}) \quad = \quad
\dim(\mathbb{M}_P) \quad = \quad
\begin{cases} 
 \,\, 3 & \text{if} \quad $g = 2$, \\
 \,\, 6 & \text{if} \quad $g = 3$, \\
 \,\, 16 & \text{if} \quad $g = 7$, \\
 \,\, 2g+1 & \text{otherwise}.
\end{cases}
\end{equation}
In each case, the cone $\,\mathbb{M}_\Delta$ of honeycomb curves 
supported on $P$ attains this dimension.
\end{theorem}

Honeycomb curves are introduced in Section \ref{sec:honeycombs}.  
That section furnishes the proof of Theorem \ref{thm:dimension}. 
The connection between tropical and classical curves
will be explained in Section~\ref{sec:classicalcurves}.  The number $2g+1$ in
(\ref{eq:dimformula}) is the dimension of the classical moduli space of \emph{trigonal curves of genus $g$},
whose tropicalization is related to our stacky fan $\mathbb{M}_g^{\rm planar}$.  Our primary
source for the relevant material from algebraic geometry is the article \cite{CV} by Castryck and
Voight.  Our paper can be seen as a refined combinatorial extension of theirs.
For related recent work that incorporates also immersions of tropical curves
see Cartwright et al.~\cite{CDMY}.

We begin in Section \ref{sec:combinat} with an introduction to the relevant background from geometric combinatorics.
The objects in (\ref{eq:MPGdecompose})--(\ref{eq:Mgplanar}) are carefully defined, and we explain our algorithms for
computing these explicitly, using the software packages \topcom \cite{Ram} and 
\polymake~\cite{AGHJLPR,GJ}.

Our main results in this paper are
Theorems  \ref{thm:planequartics}, \ref{thm:hypg3}, \ref{thm:genus4}, and \ref{thm:genus5}. 
These concern $g=3,4,5$ and they are presented in
Sections \ref{sec:genus3} through \ref{sec:fivesix}. 
The proofs of these theorems rely on the computer calculations that are described in 
Section \ref{sec:combinat}.
In Section \ref{sec:genus3} we study plane quartics
as in Figure \ref{figure:type3_case1intro}.  
Their Newton polygon is the size four triangle $T_4$.
This models non-hyperelliptic genus $3$ curves in their canonical embedding.  
We compute the space $\mathbb{M}_{T_4}$.
 Four of the five trivalent graphs
of genus $3$ are realized by smooth tropical plane curves.  

Section \ref{sec:hyperelliptic} is devoted to hyperelliptic curves.  We show that all metric graphs arising from hyperelliptic polygons of given
genus arise from a single polygon, namely, the hyperelliptic triangle.  We determine the space $\mathbb{M}^{\rm planar}_{3,\rm{hyp}}$, which
together with $\mathbb{M}_{T_4}$ gives $\mathbb{M}^{\rm planar}_3$. Section \ref{sec:genus4} deals with curves of genus $g =
4$. Here (\ref{eq:Mgplanar}) is a union over four polygons, and precisely $13$ of the $17$ trivalent graphs $G$ are
realized in (\ref{eq:MPdecompose}).  The dimensions of the cones $\mathbb{M}_{P,G}$ range between $4$ and $9$.  In
Section \ref{sec:fivesix} we study curves of genus $g = 5$.  Here $38$ of the $71$ trivalent graphs are realizable. Some
others are ruled out by the sprawling condition in Proposition \ref{prop:sprawling}.  We end with a brief discussion of
$g \geq 6$ and some open questions.

\section{Combinatorics and Computations}
\label{sec:combinat}
 
The methodology of this paper is computations in geometric combinatorics. In this section we fix
notation, supply definitions, present  algorithms, and give some core results.  For
additional background, the reader is referred to the book by De Loera, Rambau, and Santos~\cite{DRS}.

Let $P$ be a lattice polygon, and let $A=P\cap\Z^2$ be the set of lattice points in $P$.  Any
function $h:A\to\R$ is identified with a tropical polynomial with Newton polygon $P$, namely,
\[
H(x,y) \,\,\,\, = \,\,\,  \bigoplus_{(i,j) \in A } h(i,j) \odot x^i \odot y^j.
\]
The tropical curve $C$ defined by this min-plus polynomial consists of all points $(x,y) \in
\mathbb{R}^2$ for which the minimum among the quantities $ i \cdot x + j \cdot y + h(i,j)$ is
attained at least twice as $(i,j)$ runs over $A$.  The curve $C$ is dual to the \emph{regular
  subdivision} $\Delta$ of $A$ defined by $h$.  
  To construct $\Delta$, we lift each lattice point $a\in A$ to the height $h(a)$, then take the lower
convex hull of the lifted points in $\R^3$. Finally, we project back to $\R^2$ by omitting the
height.  The \emph{maximal cells} are the images of the facets of the lower convex hull under the
projection.  The set of all height functions $h $ which induce the same subdivision $\Delta$ is a
relatively open polyhedral cone in $\mathbb{R}^A$.  This is called the \emph{secondary cone} and
is denoted $\Sigma(\Delta)$.  The collection of all secondary cones $\Sigma(\Delta)$ is a complete
polyhedral fan in $\R^A$, the \emph{secondary fan} of $A$.

A subdivision $\Delta$ is a \emph{triangulation} if all maximal cells are triangles.  
The maximal cones in the secondary fan $\Sigma(\Delta)$ 
correspond to the regular triangulations $\Delta$ of $A$.
Such a cone is the product of a pointed cone of dimension $\#A-3$
and a $3$-dimensional subspace of $\R^A$.

We are interested in regular triangulations $\Delta$ of $P$
that are \emph{unimodular}. This means that each
triangle in $\Delta$ has area $1/2$, or, equivalently,
that every point in $A = P \cap \Z^2$ is a vertex of $\Delta$.
We  derive an inequality representation for the secondary cone $\Sigma(\Delta)$ as follows.
Consider any four points 
$a = (a_1,a_2)$, $b = (b_1,b_2)$,
$c = (c_1,c_2)$ and $d= (d_1,d_2)$ in $A$ such that the triples
$(c,b,a)$ and $(b,c,d)$ are clockwise oriented triangles
of $\Delta$. Then we require
\begin{equation}
\label{eq:4x4det}
\det
\begin{pmatrix}
  1 & 1 & 1 & 1 \\
  a_1 & b_1 & c_1 & d_1 \\
  a_2 & b_2 & c_2 & d_2 \\
  h(a) & h(b) & h(c) & h(d) 
  \end{pmatrix}
 \,\, \geq \,\, 0 .
\end{equation}
This is a linear inequality for $h \in \R^A$.  It can be viewed as a ``flip condition'', determining which of the two diagonals of a quadrilateral are in the subdivision.  We have one such inequality for each interior edge
$bc$ of $\Delta$.  The set of solutions to these linear inequalities is the secondary cone
$\Sigma(\Delta)$.  From this it follows that the linearity space $ \Sigma(\Delta) \cap -
\Sigma(\Delta) $ of the secondary cone is $3$-dimensional. It is the space ${\rm Lin}(A)$ of
functions $h \in \R^A$ that are restrictions of affine-linear functions on $\R^2$.  We usually
identify $\Sigma(A)$ with its image in $\R^A /{\rm Lin}(A)$, which is a pointed cone of dimension
$\#A-3$.  That pointed cone has finitely many rays and we represent these by vectors in $\R^A$.

Suppose that $\Delta$ has $E$ interior edges and $g$ interior vertices.
We consider two linear maps
\begin{equation}
\label{eq:composition}
 \R^A \,\,\buildrel{\lambda}\over{\longrightarrow} \,\, \R^E \,\, 
\buildrel{\kappa}\over{\longrightarrow}  \,\, \R^{3g-3} . 
\end{equation}
The map $\lambda$ takes $h$ and outputs the vector whose $bc$-coordinate equals
 (\ref{eq:4x4det}).  
This determinant is nonnegative: it is precisely the length of
the edge of the tropical curve $C$ that is dual to~$bc$.
Hence $\lambda(h)$ is the vector whose coordinates are the
lengths of the bounded edges of~$C$, and $\kappa(\lambda(h))$ is the vector whose $3g-3$ coordinates are the lengths of the edges of the skeleton~$G$.  

\begin{remark} \rm
The (lattice) length of an edge of $C$ with slope $p/q$, where $p,q$ are relatively prime integers, is the Euclidean length of the edge divided by $\sqrt{p^2+q^2}$.  This lets one quickly read off 
the lengths from a picture of $C$ without having to compute the
determinant (\ref{eq:4x4det}).
\end{remark}  

 Each edge $e$ of the skeleton $G$ is a concatenation of edges
of $C$.  The second map $\kappa$ adds up the corresponding lengths.  Thus the composition
\eqref{eq:composition} is the linear map with $e^{th}$ coordinate
\[
\qquad (\kappa \circ \lambda)(h)_{e} \,\,\, = \sum_{bc\,:\, \text{the dual of } bc \atop
  \text{contributes to } e} \!\!\!  \lambda(h)_{bc} \qquad \qquad \text{for all edges } e \text{ of
} G.
\]
By definition, the secondary cone is mapped into the nonnegative orthant under $\lambda$. Hence
\begin{equation}
\label{eq:composition2}
\Sigma(\Delta) \,\,\buildrel{\lambda}\over{\longrightarrow} \,\, \R_{\geq 0}^E \,\, 
\buildrel{\kappa}\over{\longrightarrow}  \,\, \R_{\geq 0}^{3g-3} . 
\end{equation}
Our discussion implies the following result on the cone of metric graphs arising from $\Delta$:

\begin{proposition}
  The cone $\,\mathbb{M}_\Delta$ is the image of the secondary cone $\Sigma(\Delta)$ under $\kappa
  \circ \lambda$.
\end{proposition}

Given any lattice polygon $P$, we seek to compute the moduli space $\mathbb{M}_P$ via the
decompositions in \eqref{eq:MPdecompose}.  Our line of attack towards that goal can now be
summarized as follows:
\begin{enumerate}
\vspace{-0.05in}
\item \label{step:topcom} compute all regular unimodular triangulations 
 of $A = P \cap \Z^2$ up to symmetry;
 \vspace{-0.1in}
\item \label{step:buckets} sort the triangulations into buckets,
  one for each trivalent graph $G$ of genus $g$;
   \vspace{-0.1in}
  \item \label{step:sec} for each triangulation $\Delta$ with skeleton
     $G$, compute its secondary cone 
  $\Sigma(\Delta) \subset \R^A$;
   \vspace{-0.1in}
\item \label{step:moduli} for each secondary cone $\Sigma(\Delta)$, compute its image
  $\mathbb{M}_\Delta$ in the moduli space $\mathbb{M}_g$ via (\ref{eq:composition2});
   \vspace{-0.1in}
\item \label{step:merge} merge the results to get the fans $\mathbb{M}_{P,G} \subset \R^{3g-3}$ in
  \eqref{eq:MPGdecompose} and the moduli space $\mathbb{M}_P$ in~\eqref{eq:MPdecompose}.
\end{enumerate}

Step \ref{step:topcom} is based on computing the secondary fan of $A$.  There are two different
approaches to doing this.  The first, more direct, method is implemented in \gfan \cite{Jen}.  It
starts out with one regular triangulation of $\Delta$, e.g. a placing triangulation arising from a
fixed ordering of~$A$.  This comes with an inequality description for $\Sigma(\Delta)$, as in
\eqref{eq:4x4det}. From this, \gfan computes the rays and the facets of $\Sigma(\Delta)$.  Then
\gfan proceeds to an adjacent secondary cone $\Sigma(\Delta')$ by producing a new height function
from traversing a facet of $\Sigma(\Delta)$.  Iterating this process results in a
breadth-first-search through the edge graph of the \emph{secondary polytope} of $A$.

The second method starts out the same.  But it passes from $\Delta$ to a neighboring triangulation
$\Delta'$ that need not be regular.  It simply performs a purely combinatorial restructuring known
as a \emph{bistellar flip}.  The resulting breadth-first search is implemented in \topcom
\cite{Ram}.  Note that a bistellar flip corresponds to inverting the sign in one of the inequalities in (\ref{eq:4x4det}).

Neither algorithm is generally superior to the other, and sometimes it is difficult to predict which
one will perform better.  The flip algorithm may suffer from wasting time by also computing
non-regular triangulations, while the polyhedral algorithm is genuinely costly since it employs
exact rational arithmetic.  The flip algorithm also uses exact coordinates, but only in a
preprocessing step which encodes the point configuration as an oriented matroid.  Both algorithms
can be modified to enumerate all regular \emph{unimodular} triangulations up to symmetry only.  For
our particular planar instances, we found \topcom to be more powerful.

We start Step \ref{step:buckets} by computing
the dual graph of a given $\Delta$. The nodes are the triangles and the edges record incidence.
Hence each node has degree $1$, $2$ or $3$.  We then recursively
delete the nodes of degree~$1$.  Next, we recursively contract edges which are incident with a node
of degree $2$.  The resulting trivalent graph $G$ is the skeleton of $\Delta$.  It often has loops
and multiple edges.  In this process we keep track of the history of all deletions and
contractions.

Steps~\ref{step:sec} and \ref{step:moduli} are carried out using \polymake \cite{GJ}.  Here the buckets or even the
individual triangulations can be treated in parallel.  The secondary cone $\Sigma(\Delta)$ is defined in $\R^A$ by the
linear inequalities $\lambda(h) \geq 0$ in \eqref{eq:4x4det}.  From this we compute the facets and rays of
$\Sigma(\Delta)$. This is essentially a convex hull computation. In order to get unique rays modulo ${\rm Lin}(A)$, we
fix $h= 0$ on the three vertices of one particular triangle.  Since the cones are rather small, the choice of the convex
hull algorithm does not matter much.  For details on state-of-the-art convex hull computations and an up-to-date
description of the \polymake system see \cite{AGHJLPR}.

For Step~\ref{step:moduli}, we apply the linear map $\kappa \circ \lambda$ to all rays of the secondary cone
$\Sigma(\Delta)$. Their images are vectors in $\R^{3g-3}$ that span the moduli cone $\mathbb{M}_\Delta = (\kappa \circ
\lambda)(\Sigma(\Delta))$.  Via a convex hull computation as above, we compute all the rays and facets of
$\mathbb{M}_\Delta $.

The cones $\mathbb{M}_\Delta$ are generally not full-dimensional in $\R^{3g-3}$. The points in the
relative interior are images of interior points of $\Sigma(\Delta)$. Only these represent smooth
tropical curves. However, it can happen that another cone $\mathbb{M}_{\Delta'}$ is a face of
$\mathbb{M}_\Delta$.  In that case, the metric graphs in the relative interior of that face are also
realizable by smooth tropical curves.

Step \ref{step:merge} has not been fully automatized yet, but we carry it out in a case-by-case
manner. This will be described in detail for
curves of genus $g=3$ in Sections \ref{sec:genus3} and \ref{sec:hyperelliptic}.

\smallskip

We now come to the question of what lattice polygons $P$ should be the input for 
Step \ref{step:topcom}. Our point of departure towards answering that question is the following finiteness result.

\begin{proposition} \label{prop:finite}
For every fixed genus $g \geq 1$, there are only finitely many
lattice polygons $P$ with $g$ interior lattice points,
up to integer affine isomorphisms in $\Z^2$.
\end{proposition}

\begin{proof}[Proof and Discussion]
 Scott \cite{Sco} proved that $\#(\partial P\cap\Z^2)
\leq 2g+7$,  and this bound is sharp.  This means that the number of interior lattice points yields a bound on the
total number of lattice points in $P$.  This result was generalized to arbitrary dimensions by
Hensley~\cite{Hen}.  Lagarias and Ziegler \cite{LZ}
improved Hensley's bound and further observed that there are only finitely many lattice polytopes 
with a given total number of lattice points, up to unimodular equivalence \cite[Theorem~2]{LZ}.  Castryck  \cite{Ca} gave an algorithm for finding all lattice polygons of a given genus, along with the number of lattice polygons for each genus up to 30.
We remark that the assumption $g\ge 1$ is essential, as there are lattice triangles of
arbitrarily large area and without any interior lattice point.
\end{proof}

Proposition \ref{prop:finite} ensures that the union in (\ref{eq:Mgplanar})
is finite. However, from the full list of polygons $P$ with
$g$ interior lattice points, only very few will be needed to construct
$\mathbb{M}_{g}^{\rm planar}$. To show this, and to illustrate the 
concepts seen so far,
we now discuss our spaces for $g \leq 2$.

\begin{example} 
\label{eq:genus1} \rm
For $g=1$, only one polygon $P$ is needed in (\ref{eq:Mgplanar}), and
only one triangulation $\Delta$ is needed in (\ref{eq:MPdecompose}).
We take $P = \conv\{(0,0),(0,3),(3,0)\}$, since
  every smooth genus $1$ curve is a plane cubic, and
we let $\Delta$ be the honeycomb triangulation from Section \ref{sec:honeycombs}.
The skeleton $G$ is a cycle whose length is the tropical j-invariant \cite[\S 7.1]{BPR}.
We can summarize this as follows:
 \begin{equation}
 \label{eq:genus1equalities}
 \mathbb{M}_\Delta \,\,=\,\, \mathbb{M}_{P,G} \,\, = \,\,
 \mathbb{M}_P \,\, = \,\, \mathbb{M}_1^{\rm planar} \,\, = \,\, 
 \mathbb{M}_1 \,\, = \,\, \R_{\geq 0}.
 \end{equation}
All inclusions in
\eqref{eq:inclusiondiagram} are equalities
for this particular choice of $(P,\Delta)$.
\end{example}

 \begin{figure}[h]
\centering
\includegraphics[scale=1.2]{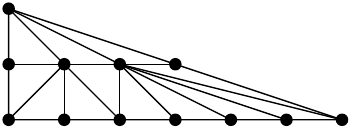} \qquad 
\includegraphics[scale=1.2]{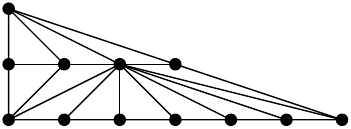} \qquad \qquad
\includegraphics[scale=1.2]{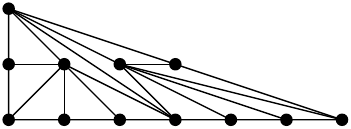} 
\caption{The triangulations $\Delta_1$, $\Delta'_1$, and $\Delta_2$}
\label{figure:g2_triangulations}
\end{figure}

\begin{example}
  \label{eq:genus2} \rm  In classical algebraic geometry, all smooth curves of genus $g = 2$ are hyperelliptic, and they can be
  realized with the Newton polygon $P = \text{conv}\{(0,0),(0,2),(6,0)\}$.  There are two trivalent graphs of genus $2$,
  namely, the \emph{theta graph} $G_1=\raisebox{-.18\height}{\includegraphics[scale=0.3]{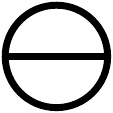}}$  and the \emph{dumbbell graph} $G_2={\includegraphics[scale=0.3]{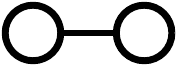}}$.  The moduli space
  $\mathbb{M}_2$ consists of two quotients of the orthant $\R^3_{\geq 0}$, one for each graph, glued together.  For
  nice drawings see Figures 3 and 4 in \cite{Chan}.  Figure \ref{figure:g2_triangulations} shows three unimodular triangulations
  $\Delta_1$, $\Delta'_1$, and $\Delta_2$ of $P$ such that almost all metric graphs in $\mathbb{M}_2$ are realized
  by a smooth tropical curve $C$ dual to $\Delta_1$, $\Delta'_1$, or $\Delta_2$.  We say ``almost all'' because   here
  the three edges of $G_1$ cannot have all the same length 
  \cite[Proposition 4.7]{CDMY}.
    The triangulations $\Delta_1$ and $\Delta_1'$ both give $G_1$ as a skeleton.  If  $a\geq b\geq c$ denote the edge lengths on $G_1$, then the curves dual to $\Delta_1$ realize all metrics with $a\geq b>c$, and the curves dual to $\Delta_1'$ realize all metrics with $a>b=c$.  The triangulation $\Delta_2$ gives $G_2$ as a skeleton, and the curves dual to it achieve all possible metrics.  Since our $3$-dimensional cones are closed by definition, 
    \begin{equation}
    \label{eq:genus2equalities}
    \left(\mathbb{M}_{\Delta_1} \cup \, \mathbb{M}_{\Delta'_1}\right)  \cup \,\mathbb{M}_{\Delta_2}  \,\,=\,\, 
    \mathbb{M}_{P,G_1}   \,\cup \,  \mathbb{M}_{P,G_2}  \,\, = \,\,
    \mathbb{M}_P \,\, = \,\, \mathbb{M}_2^{\rm planar} \,\, = \,\, 
    \mathbb{M}_2  \,\, = \,\, 
    \hbox{\cite[Figure 3]{Chan}.}
  \end{equation}
    In  Section \ref{sec:hyperelliptic} we extend this analysis to  hyperelliptic curves of genus $g \geq 3$.
  The graphs $G_1$ and $G_2$ represent the \emph{chains} for $g=2$.
  For information on hyperelliptic skeletons      see~\cite{Chan2}.
    \end{example}

With $g = 1,2$ out of the way, we now assume $g \geq 3$.  We follow the approach of Castryck and
Voight \cite{CV} in constructing polygons $P$ that suffice for the union \eqref{eq:Mgplanar}.  We write
$P_{\interior}$ for the convex hull of the $g$ interior lattice points of~$P$. This is the \emph{interior hull} of~$P$.  The relationship between the polygons $P$ and $P_{\rm int}$ is studied in
polyhedral adjunction theory  \cite{DHNP}.

\begin{lemma}
  Let $P \subseteq Q$ be lattice polygons with $P_{\interior} = Q_{\interior}$.  Then $\mathbb{M}_P$
  is contained in $\mathbb{M}_Q$.
\end{lemma}

\begin{proof}
  By \cite[Lemma 4.3.5]{DRS}, a triangulation $\Delta$ of any point set $S$
can be extended to a triangulation $\Delta'$ of any superset $S' \supset S$.
 If $\Delta$ is regular then so is $\Delta'$. Applying this result to a
regular triangulation of $P$ which uses all lattice points in $P$ yields
a regular triangulation of $Q$ which uses all lattice points in $Q$.  The
triangulations of a lattice polygon which use all lattice points are
precisely the unimodular ones.  (This is a special property of planar
triangulations.)  We conclude that every
  tropical curve $C$ dual to $\Delta$ is contained in a curve $C'$ dual to
  $\Delta'$, except for unbounded edges of $C$. The skeleton and its possible metrics remain unchanged, since $P_{\interior} = Q_{\interior}$. We therefore have the equality of moduli cones
   $\mathbb{M}_\Delta = \mathbb{M}_{\Delta'}$. The unions for $P$ and $Q$ in
  \eqref{eq:MPdecompose} show that $\mathbb{M}_P \subseteq \mathbb{M}_Q$.
  \end{proof}

This lemma shows that we only need to consider \emph{maximal polygons}, i.e. those $P$ that are maximal with respect to
inclusion for fixed $P_{\interior}$.  If $P_{\interior}$ is $2$-dimensional then this determines $P$ uniquely. Namely,
suppose that $P_{\interior} = \{(x,y) \in \R^2 \,:\, a_i x + b_i y \leq c_i \,{\rm for} \,\, i = 1,2,\ldots,s \}$, where $\gcd(a_i,b_i,c_i)=1$ for all $i$.  Then
$P$ is the polygon $\,\{(x,y) \in \R^2 \,:\, a_i x + b_i y \leq c_i+1 \,\,{\rm for} \,\, i = 1,2,\ldots,s \}$.  If $P$
is a lattice polygon then it is a maximal lattice polygon.  However, it can happen that $P$ has non-integral
vertices. In that case, the given $P_{\rm int}$ is not the interior of any lattice polygon.

The maximal polygon $P$ is not uniquely determined by $P_{\rm int}$ when $P_{\rm int}$ is a line segment.
For each $g \geq 2$
 there are $g+2$ distinct \emph{hyperelliptic trapezoids} to be considered.  We shall see in
Theorem \ref{thm:chains} that for our purposes it suffices to use the triangle ${\rm conv}\{ (0,0), (0,2),(2g+2,0)\}$.

Here is the list of all maximal polygons we use as input for the pipeline described above.

\begin{proposition} \label{prop:twelvepolygons}
Up to isomorphism there are precisely $12$ maximal polygons $P$
such that $P_{\rm int}$ is $2$-dimensional and
$3 \leq g = \# (P_{\rm int} \cap \mathbb{Z}^2) \leq 6$.
For $g = 3$, there is a unique type, namely,
$\,T_4 = \conv\{ (0,0), (0,4),(4,0)\}$.
For $g=4$ there are three types:
\[
\begin{matrix} 
Q^{(4)}_1 = R_{3,3} \!  &=& \conv\{ (0,0),(0,3),(3,0),(3,3)\} , \qquad & & 
Q^{(4)}_2 &=& \conv \{ (0,0), (0,3),(6,0)\}, \\ & & & & 
Q^{(4)}_3 &=& \conv \{(0,2),(2,4),(4,0)\} .\\
\end{matrix}
\]
For $g=5$  there are four types of maximal polygons:
\[
\begin{matrix}   
 Q^{(5)}_1 & = & {\rm   conv} \{ (0, 0), (0, 4), (4, 2) \} , & & 
 Q^{(5)}_2 &=& {\rm conv} \{(2,0),(5,0), (0,5),(0,2)\}, \\
 Q^{(5)}_3 & = & {\rm conv} \{ (2,0), (4,2),(2,4), (0,2) \},  & &
 Q^{(5)}_4 & = & {\rm conv} \{ (0, 0), (0, 2),  (2, 0), (4, 4) \} .
\end{matrix}
\]
For $g=6$  there are four types of maximal polygons:
\[
\begin{matrix}
  Q^{(6)}_1 = T_5 \!& \!= \! & {\rm conv} \{ (0, 0), (0, 5), (5, 0) \} , & & 
  Q^{(6)}_2 & \!=\! &  {\rm conv} \{ (0,0), (0, 7), (3, 0), (3, 1) \}, \\
  Q^{(6)}_3 = R_{3,4} \! & \!=\! &\! {\rm conv} \{ (0, 0), (0, 4), (3, 0),  (3, 4) \} , & & 
  Q^{(6)}_4 & \!=\! & {\rm conv} \{ (0, 0), (0, 4), (2, 0), (4, 2) \}.
\end{matrix}
\]
\end{proposition}

The notation we use for polygons is as follows.
We write $Q^{(g)}_i$ for maximal polygons of genus $g$,
but we also use a systematic notation for families of polygons,
including the \emph{triangles}
$\,T_d = \conv \{ (0,0),(0,d),(d,0) \}$ and
the \emph{rectangles}
$\,R_{d,e} = \conv \{(0,0),(d,0),(0,e),(d,e)\}$.

Proposition \ref{prop:twelvepolygons} is found by exhaustive search, using Castryck's method in
\cite{Ca}.  We started by classifying all types of lattice polygons with precisely $g$ lattice points.  These are our
candidates for $P_{\rm int}$. For instance, for $g=5$, there are six such polygons. Four of them are the interior hulls
of the polygons $Q^{(5)}_i$ with $i=1,2,3,4$.  The other two are the triangles
\[
\conv \{ (1, 1),  (1, 4),  (2, 1) \}
\qquad \text{and} \qquad
\conv \{ (1, 1),  (2, 4),  (3, 2) \}.
\]
However, neither of these two triangles
arises as $P_{\rm int}$ for any lattice polygon $P$.

For each genus $g$, we construct the stacky fans $\mathbb{M}_g^{\text{planar}}$
by computing each of the spaces $\mathbb{M}_{Q^{(g)}_i}$ and then 
subdividing their union appropriately.
This is then augmented in Section \ref{sec:hyperelliptic}
 by the spaces $\mathbb{M}_P$
   where $P_{\text{int}}$ is not two-dimensional, but is instead a line segment.

\section{Algebraic Geometry}
\label{sec:classicalcurves}

In this section we discuss the context from algebraic geometry that lies behind our computations and
combinatorial analyses.  Let $K$ be an algebraically closed field that is complete with respect to a
surjective non-archimedean valuation $\,\val: K^* \rightarrow \mathbb{R}$.  Every smooth
complete curve $\mathcal{C}$ over $K$ defines a metric graph $G$. This is the \emph{Berkovich
  skeleton} of the analytification of $\mathcal{C}$ as in \cite{BPR}.  By our hypotheses, every
metric graph $G$ of genus $g$ arises from some curve $\mathcal{C}$ over $K$.  This defines a
surjective tropicalization map from (the $K$-valued points in) the moduli space of smooth curves of
genus $g$ to the moduli space of metric graphs of genus $g$:
\begin{equation}
\label{eq:naivetrop}
 \trop :\, \mathcal{M}_g \,\rightarrow \, \mathbb{M}_g .
 \end{equation}
 Both spaces have dimension $3g-3$ for $g \geq 2$.
 The map (\ref{eq:naivetrop}) is referred to
 as ``naive set-theoretic tropicalization'' by Abramovich, Caporaso, and Payne  \cite{ACP}.
 We point to that article and its bibliography for the proper
 moduli-theoretic settings for our combinatorial objects.
 
 Consider plane curves defined by a Laurent polynomial $f=\sum_{(i,j)\in \mathbb{Z}^2}c_{ij}x^iy^j\in K[x^{\pm},y^{\pm}]$ with Newton polygon $P$.  For $\tau$ a face of  $P$ we let $f|_\tau=\sum_{(i,j)\in \tau}c_{ij}x^iy^j$, and say that $f$ is \emph{non-degenerate} if $f|_{\tau}$ has no singularities in $(K^*)^2$ for any face $\tau$ of  $P$.  Non-degenerate polynomials are useful for studying many subjects in algebraic geometry, including singularity theory \cite{Kou}, the theory of sparse resultants \cite{GKZ}, and topology of real algebraic curves \cite{Mi}.

Let $P$ be any lattice polygon in $\R^2$ with $g$
interior lattice points. We write $\mathcal{M}_P$ for the Zariski closure  
(inside  the non-compact moduli space $\mathcal{M}_g$) of the set of curves
that appear as
non-degenerate plane curves  over $K$ with Newton polygon $P$.  This space was
introduced by Koelman \cite{Ko}.   analogy to (\ref{eq:Mgplanar}), we consider the union over all relevant~polygons:
\begin{equation}
\label{eq:MCALgplanar}
 \mathcal{M}_{g}^{\rm planar} \,\,\, := \,\,\,\,
\bigcup_P \,\mathcal{M}_P .
\end{equation}
This moduli space was introduced and studied by
Castryck and Voight in \cite{CV}. That article
was a primary source of inspiration for our study. In particular,
\cite[Theorem 2.1]{CV} determined the dimensions of the spaces
$\mathcal{M}_{g}^{\rm planar}$
for all $g$. 
Whenever we speak  about the
``dimension expected from classical algebraic geometry'', 
as we do in Theorem~\ref{thm:dimension},
this refers to the formulas for 
$\dim(\mathcal{M}_P)$ and
$\dim( \mathcal{M}_{g}^{\rm planar} )$
that were derived by Castryck and Voight.    

By the Structure Theorem for Tropical Varieties \cite[\S 3.3]{MS}, these dimensions are preserved
under the tropicalization map \eqref{eq:naivetrop}.  The images $\trop(\mathcal{M}_P)$ and
$\trop(\mathcal{M}_{g}^{\rm planar})$ are stacky fans that live inside $\mathbb{M}_g = {\rm
  trop}(\mathcal{M}_g)$ and have the expected dimension.  Furthermore, all maximal cones in ${\rm
  trop}(\mathcal{M}_P)$ have the same dimension since $\mathcal{M}_P$ is irreducible (in fact, unirational).

We summarize the objects discussed so far in a diagram
of surjections and inclusions:
\begin{equation}
\label{eq:inclusiondiagram}
 \begin{matrix}
& & & & \mathcal{M}_P  & \subseteq &  \mathcal{M}_g^{\rm planar} &  \subseteq & \mathcal{M}_g \\
& & & & \downarrow && \downarrow & & \downarrow\\
& & & & \trop(\mathcal{M}_P)  & \subseteq & 
\trop(\mathcal{M}_g^{\rm planar}) &  \subseteq & 
\trop(\mathcal{M}_g)  \smallskip \\
 & & & & \text{\rotatebox{90}{$\subseteq$}} && \text{\rotatebox{90}{$\subseteq$}} & & \text{\rotatebox{90}{$=$}} \\
\mathbb{M}_\Delta  & \,\subseteq   & \mathbb{M}_{P,G}  & \,\subseteq  \!\!\!\! &
\mathbb{M}_P  & \subseteq &  \mathbb{M}_g^{\rm planar} &  \subseteq & \mathbb{M}_g \\
\end{matrix}
\end{equation}
For $g \geq 3$, the inclusions between the second row and the third row
are strict, by a wide margin. This is the distinction between
 tropicalizations of plane curves and tropical plane curves.
One main objective of this paper is to understand how the latter sit inside the former.

For example, consider $g = 3$ and $T_4 = \conv\{(0,0),(0,4),(4,0)\}$.
Disregarding the hyperelliptic locus, 
equality holds in the second~row:
\begin{equation}
\label{eq:genus3classical}
\,\trop(\mathcal{M}_{T_4}) \,\, = \,\,
\trop(\mathcal{M}_3^{\rm planar})\,\, = \,\, \trop(\mathcal{M}_3)\,\, =\,\, \mathbb{M}_3.
\end{equation}
This is the stacky fan in \cite[Figure 1]{Chan}.  The space
$\mathbb{M}_{T_4}=\mathbb{M}^{\rm{planar}}_{3,{\rm nonhyp}}$ of tropical plane quartics is also six-dimensional, but it 
is smaller. It fills up less than 30\% of the curves in $\mathbb{M}_3$; see Corollary
\ref{cor:g3:probability}. Most metric graphs of genus $3$ do {\bf not} come from plane quartics.

For $g = 4$, the canonical curve is a complete intersection of a quadric surface with a cubic surface.  If 
the quadric is smooth then we get a curve of bidegree $(3,3)$ in $\PP^1 \times \PP^1$.
This  leads to the Newton polygon $R_{3,3} = \conv\{(0,0),(3,0),(0,3),(3,3)\}$.  Singular surfaces lead to 
 families of genus $4$ curves of codimension $1$ and $2$ that are supported on
   two other polygons \cite[\S 6]{CV}.
As we shall see in Theorem \ref{thm:genus4}, 
$\mathbb{M}_{P}$  has the expected dimension for each of the three polygons $P$.
Furthermore, $\mathbb{M}^{\rm planar}_4$ is
strictly contained in $\trop(\mathcal{M}_4^{\rm planar})$.  Detailed computations that
reveal our spaces for $g=3,4,5$ are presented in Sections \ref{sec:genus3}, \ref{sec:hyperelliptic}, 
\ref{sec:genus4}, and \ref{sec:fivesix}.

\smallskip

We close this section by returning once more to classical algebraic geometry. Let $\mathcal{T}_g$
denote the \emph{trigonal locus} in the moduli space $\mathcal{M}_g$.  It is well known that
$\mathcal{T}_g$ is an irreducible subvariety of dimension $2g+1$ when $g \geq 5$. For a proof see
\cite[Proposition 2.3]{FL}.  A recent theorem of Ma \cite{Ma2} states that
$\mathcal{T}_g$ is a rational variety for all $g$.

We note that Ma's work, as well as the classical approaches to trigonal curves, are based on the
fact that canonical trigonal curves of genus $g$ are realized by a certain special polygon $P$. This
is either the rectangle in \eqref{eq:trapezoid1} or the trapezoid in \eqref{eq:trapezoid2}.  These
polygons appear in \cite[Section 12]{CV}, where they are used to argue that $\mathcal{T}_g$ defines
one of the irreducible components of $\mathcal{M}_g^{\rm planar}$, namely, $\mathcal{M}_P$.  The same
$P$ appear in the next section, where they serve to prove one inequality on the dimension in Theorem
\ref{thm:dimension}.  The combinatorial moduli space $\mathbb{M}_P$ is full-dimensional in the
tropicalization of the trigonal locus. The latter space, denoted $\trop(\mathcal{T}_g)$, is contained in
the space of trigonal metric graphs, by Baker's Specialization Lemma \cite[\S 2]{Baker}.

In general, $\mathcal{M}_g^{\rm planar}$ has many
irreducible components other than the trigonal locus $\mathcal{T}_g$.
As a consequence, there are many skeleta in
$\mathbb{M}_g^{\rm planar}$ that are not trigonal in the sense
of metric graph theory. This is seen clearly in the top dimension for $ g= 7$,
where  $\dim(\mathcal{T}_7) = 15$ but
$\dim(\mathcal{M}_7^{\rm planar}) = 16$.
The number $16$ comes from the family of trinodal sextics in
\cite[\S 12]{CV}.

\section{Honeycombs}
\label{sec:honeycombs}

We now prove Theorem~\ref{thm:dimension}.  This will be done using the special family of
\emph{honeycomb curves}.
The material in this section is purely combinatorial.  No algebraic geometry will be required.

We begin by defining the polygons that admit a honeycomb triangulation.  These polygons
depend on four integer parameters $a,b,c$ and $d$ that satisfy the constraints
\begin{equation}
\label{eq:sixineq} 0 \,\leq \, c \,\leq \,a,b \, \leq \, d \,\leq \, a+b. 
\end{equation}
To such a quadruple $(a,b,c,d)$, we associate the polygon
\[
H_{a,b,c,d} \quad = \quad \bigl\{ (x,y) \in \R^2\,:\, 0 \leq x \leq a \,\,{\rm
  and} \,\, 0 \leq y \leq b \,\, {\rm and} \,\, c \leq x+y \leq d
\bigr\}.
\]
If all six inequalities in (\ref{eq:sixineq}) are non-redundant then $H_{a,b,c,d}$ is a hexagon. Otherwise it
can be  a pentagon, quadrangle, triangle, segment, or just a point.  The number of lattice
points~is
\[
\#(H_{a,b,c,d} \,\cap \, \Z^2) \,\,\, = \,\,\,
ad+bd-\frac{1}{2}(a^2+b^2+c^2+d^2) +\frac{1}{2}(a+b-c+d) + 1,
\]
and, by Pick's Theorem, the number of interior lattice points is
\[
g \,\, = \,\, \#((H_{a,b,c,d})_{\interior} \,\cap \, \Z^2) \,\, = \,\,
ad+bd-\frac{1}{2}(a^2+b^2+c^2+d^2) 
-\frac{1}{2}(a+b-c+d) + 1 .\quad
\]
The \emph{honeycomb triangulation} $\Delta$ subdivides $H_{a,b,c,d}$ 
 into $2ad+2bd-(a^2+b^2+c^2+d^2) $ unit triangles.
It is obtained by slicing $H_{a,b,c,d}$ with the vertical lines $\{x
= i\}$ for $0 < i < a$, the horizontal lines $\{y=j\}$ for $0 < j <
b$, and the diagonal lines $\{x+y = k \}$ for $c < k < d$.  The
tropical curves $C$ dual to $\Delta$ look like honeycombs, as 
seen in the middle of Figure \ref{figure:honeycomb_picture}. The
corresponding skeleta $G$ are called \emph{honeycomb graphs}.

\begin{figure}[h]
\centering
\includegraphics[scale=1.8]{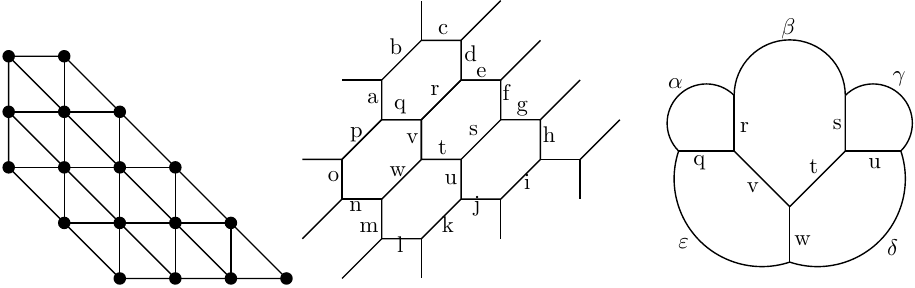}
\caption{The honeycomb triangulation of $H_{5,4,2,5}$, the tropical curve, and its skeleton}
\label{figure:honeycomb_picture}
\end{figure}
 
If $P = H_{a,b,c,d}$  then
its interior $P_{\interior}$ is a honeycomb polygon as well.
Indeed,  a translate of $P_{\interior}$
 can be obtained from $P$ by decreasing the values of
 $a,b,c,d$ by an appropriate amount.
 
 \begin{example}
 \label{ex:fivehexagons}
 \rm
 Let $P = H_{5,4,2,5}$.  Note that $P_{\interior} = H_{3,3,1,2} + (1,1)$.  The honeycomb triangulation $\Delta$ of $P$ is illustrated in Figure \ref{figure:honeycomb_picture}, together with a dual tropical curve and its skeleton.  The bounded edge lengths in the tropical curve are labelled $a$ through $w$.  These lengths induce the edge lengths on the skeleton, via the formulas
 $\alpha=a+b+c+d$, $\beta=e+f$, $\gamma=g+h+i+j$, $\delta=k+l+m$, and $\varepsilon = n+o+p$.
This is the map 
$\kappa: \R^{23} \rightarrow \R^{12}$ in~(\ref{eq:composition}).

The  cone $\lambda(\Sigma(\Delta)) \subset \R^{23}_{\geq 0}$ has dimension $13$ and is defined by the
ten linear equations
\begin{equation}
\label{eq:tenten}
 \begin{matrix}
 \underline{a}+b=d+r & \quad \underline{e}+f=t+v & \quad \underline{g}+h = j+ u & 
 \quad \underline{k}+l  = t+w & \quad \underline{n}+o = q+v \\
 \underline{b}+c=r+q  & \quad \underline{f}+s-v-r &\quad  \underline{h}+i = u+ s & 
 \quad \underline{l}+m = t+u  & \quad \underline{o}+p = v+w \\
\end{matrix}
\end{equation}
It has $31$ extreme rays. Among their images under $\kappa$, only $17$
are extreme rays of the moduli cone $\mathbb{M}_\Delta$.
We find that $\mathbb{M}_\Delta = \kappa(\lambda(\Sigma(\Delta)))$ has codimension one  in $\R^{12}$.
It is defined by  the non-negativity of the $12$ edge lengths,
  by the equality  $\beta=t+v$, and by the inequalities 
$$ 
\begin{matrix}
q+r\leq \alpha, \quad s+u\leq \gamma, \quad \max\{t+w,t+u\}\leq \delta\leq 2t+u+w, \\
 \max\{q+v,v+w\}\leq \varepsilon\leq q+2v+w, \quad r\leq s+t, \quad s\leq r+v.
 \end{matrix}
 $$
The number ${\rm dim}(\mathbb{M}_\Delta) = 11$ is explained by the following lemma.

 
 \end{example}
 
 \begin{lemma}
 \label{lem:honeydim}
Let $\Delta$ be the honeycomb triangulation of  $P  = H_{a,b,c,d}$. Then
\[
\dim(\mathbb{M}_\Delta) \,\, = \,\, \# (P_{\interior} \cap \Z^2) \,+\,
\#( \partial P_{\interior} \cap \Z^2) \,+\, \# \vertices(P_{\rm
  int}) \,- \, 3.
\]
\end{lemma}

\begin{proof}
The honeycomb graph $G$ consists of
$g = \# (P_{\interior} \cap \Z^2)$ hexagons. 
The hexagons associated with lattice points on the 
boundary of $P_{\interior}$ have vertices that are $2$-valent in $G$.
Such $2$-valent vertices get removed, so these boundary
hexagons become cycles with fewer than six edges.
In the orthant $\R_{\geq 0}^{3g-3}$ of all metrics on $G$,
we consider the subcone of metrics $\mathbb{M}_\Delta$ that arise from $\Delta$.
This is the image under $\kappa$ of the transformed secondary cone
$\lambda(\Sigma(\Delta))$. 

The cone $\lambda(\Sigma(\Delta))$ is defined in $\R^E_{\geq 0}$ by
$2g$ linearly independent linear equations, namely, two per hexagon.
These  state that the sum of the lengths of any two adjacent edges 
equals that of the opposite sum. For instance, in Example~\ref{ex:fivehexagons},
each of the five hexagons contributes two linear equations,
listed in the columns of (\ref{eq:tenten}). These equations can be chosen to 
have distinct leading terms,
underlined in (\ref{eq:tenten}). In particular, they are linearly independent.

Now, under the elimination process that represents the projection $\kappa$, we retain
\begin{itemize}
\item[(i)] two linear equations for each lattice point in the interior of $P_{\interior}$;
  \vspace{-0.08in}
\item[(ii)] one linear equation for each lattice point in the relative interior 
of an edge of $P_{\interior}$;
 \vspace{-0.08in}
\item[(iii)] no linear inequality from the vertices of $P_{\interior}$.
\end{itemize}
That these equations are independent follows from the triangular structure, as in (\ref{eq:tenten}).
Inside the linear space defined by these equations, the moduli cone $\mathbb{M}_\Delta$ is defined by
various linear inequalities
all of which, are strict when the graph $G$ comes from a tropical curve $C$ in
the interior of $\Sigma(\Delta)$.
This implies that the codimension of  $\mathbb{M}_\Delta$
inside  the orthant $\R_{\geq 0}^{3g-3}$  equals
\begin{equation}
\label{eq:codimformula}
 \codim(\mathbb{M}_\Delta) \,\,\, = \,\,\,\,
\left(\#( \partial P_{\interior} \cap \Z^2) \,-\, \# \vertices(P_{\interior})\right)
\,+\,  2 \cdot \# (\interior(P_{\interior}) \cap \Z^2). 
\end{equation}
This expression can be rewritten as
\[
\,g \,+\,  \# (\interior(P_{\interior}) \cap \Z^2) \,-\, \# \vertices(P_{\interior})  
\quad = \quad 2g \,-\,
 \#( \partial P_{\interior} \cap \Z^2)  \,-\, \# \vertices(P_{\rm
   int})  .
\]
Subtracting this codimension from $3g-3$, we obtain the
desired formula.
\end{proof}

\begin{proof}[Proof of Theorem~\ref{thm:dimension}]
  For the classical moduli space $\mathcal{M}_g^{\rm planar}$, the formula  \eqref{eq:dimformula} was proved
  in \cite{CV}.  That dimension is preserved under tropicalization.  The inclusion of $\mathbb{M}_g^{\rm planar}$ in
  $ \trop(\mathcal{M}_g^{\rm planar}) $, in \eqref{eq:inclusiondiagram}, implies that the right-hand side in
  \eqref{eq:dimformula} is an upper bound on $\dim(\mathbb{M}_g^{\rm planar})$.

  To prove the lower bound, we choose $P$ to be a specific honeycomb polygon with honeycomb triangulation $\Delta$. Our choice depends on the parity of the
  genus $g$. If $g = 2h$ is even then we take the rectangle
\begin{equation}
\label{eq:trapezoid1}
 R_{3,h+1} \,=\,  H_{3,h+1,0,h+4} \,=\, \conv\{ (0,0),(0,h+1),(3,0),(3,h+1) \}. 
 \end{equation}
The interior hull of $R_{3,h+1}$ is the rectangle
\[
(R_{3,h+1})_{\interior} 
\,\,=\, \conv \{ (1,1), (1,h), (2,1), (2,h)\}  \,\, \cong \,\, R_{1,h-1}.
\]
All $g = 2h$ lattice points of this polygon lie on the boundary.
From Lemma \ref{lem:honeydim}, we see that
${\dim}(\mathbb{M}_\Delta) = g + g + 4-3 = 2g+1$.
If $g = 2h+1$ is odd then we take the trapezoid
\begin{equation}
\label{eq:trapezoid2}
 H_{3,h+3,0,h+3} \,=\,  \conv\{ (0,0),(0,h+3),(3,0),(3,h) \}. 
 \end{equation}
The convex hull of the interior lattice points in
$H_{3,h+3,0,h+3}$ is the trapezoid
\[ (H_{3,h+3,0,h+3})_{\interior} \,=\, \conv \{ (1,1), (1,h+1), (2,1), (2,h) \}. \]
All $g = 2h+1$ lattice points of this polygon lie on its boundary,
and again
${\dim}(\mathbb{M}_\Delta)  = 2g+1$.

For all $g \geq 4$ with $g \not= 7$, this
matches the upper bound obtained from \cite{CV}.
We  conclude that
$\,\dim(\mathbb{M}_P) = \dim(\mathbb{M}_g) = 2g+1\,$ holds
in all of these cases. For $g = 7$ we take $P = H_{4,4,2,6}$.
Then $P_{\interior}$ is a hexagon with $g =7$ lattice points.
From Lemma \ref{lem:honeydim}, 
we find $ \dim(M_\Delta) =  7 + 6 + 6 - 3 =  16 $,  so this matches the upper bound.
Finally, for $g = 3$, we will see ${\rm dim}(M_{T_4}) = 6$ 
in Section \ref{sec:genus3}. The case $g=2$ follows from the discussion in Example \ref{eq:genus2}.
\end{proof}

There are two special families of honeycomb curves: those
arising from  the triangles $T_d $ for $ d \geq 4$ and 
rectangles $R_{d,e}$  for $d,e \geq 3$.
The triangle $T_d$ corresponds to curves
of degree $d$ in the projective plane $\PP^2$.
Their genus is $g = (d-1)(d-2)/2$.
The case $d=4, g=3$ will be our topic in
Section \ref{sec:genus3}.
The rectangle $R_{d,e}$ corresponds to curves of
bidegree $(d,e)$ in $\PP^1 \times \PP^1$.
Their genus is $g = (d-1)(e-1)$.
The case $d=e=3, g=4$ appears in
 Section \ref{sec:genus4}.

\begin{proposition}
\label{prop:trianglerectangle}
Let  $P$ be the triangle $T_d$ with $d \geq 4$ or the rectangle $R_{d,e}$ with
$d,e \geq 3$.
The moduli space $\mathbb{M}_P$ of tropical plane curves
has the expected dimension inside $\mathbb{M}_g$, namely,
\[
\dim(\mathbb{M}_{T_d}) \,= \, \frac{1}{2}d^2 + \frac{3}{2} d - 8
\quad \hbox{and} \quad
\codim(\mathbb{M}_{T_d}) \,= \, (d-2)(d-4) ,
\,\,\hbox{whereas}
\]
\[
\dim(\mathbb{M}_{R_{d,e}}) \, = \, de + d +e - 6 
\quad \hbox{and} \quad
\codim(\mathbb{M}_{R_{d,e}}) \, = \, 2(de-2d-2e+3).
\]
In particular, the honeycomb triangulation
defines a cone $\mathbb{M}_\Delta$ of this maximal dimension.
\end{proposition}

\begin{proof}
For our standard triangles and rectangles, the
 formula (\ref{eq:codimformula}) implies
\[
\begin{matrix}
\codim(\mathbb{M}_{T_d}) & = &
 3(d-3) \,\, -\,\,3  \,\,+\,\, 2 \cdot \frac{1}{2} (d-4)(d-5),  \\
\codim(\mathbb{M}_{R_{d,e}}) & = &
2((d-2)+(e-2)) \,- 4 \, + \,2 \cdot ( d-3)(e-3).
\end{matrix}
\]
Subtracting from $3g-3 = \dim(\mathbb{M}_g)$, we get
the desired formulas for $\dim(\mathbb{M}_P)$.
\end{proof}

The above dimensions are those expected from algebraic geometry.  Plane curves with Newton polygon $T_d$ form a
projective space of dimension $\frac{1}{2}(d+2)(d+1)-1$ on which the $8$-dimensional group ${\rm PGL}(3)$ acts
effectively, while those with $R_{d,e}$ form a space of dimension $(d+1)(e+1)-1$ on which the $6$-dimensional group
${\rm PGL}(2)^2 $ acts effectively.  In each case, $\dim(\mathcal{M}_P)$ equals the dimension of the family of all
curves minus the dimension of the group.

\section{Genus Three}
\label{sec:genus3}

In classical algebraic geometry, all non-hyperelliptic smooth curves
of genus $3$ are plane quartics.
Their Newton polygon $T_4 = \conv \{(0,0),(0,4),(4,0)\}$ is
the unique maximal polygon with $g=3$ in
Proposition  \ref{prop:twelvepolygons}.
 In this section, we compute the moduli space $ \mathbb{M}_{T_4}$,
and we characterize the dense subset of metric
graphs that are realized by smooth tropical quartics.
 In the next section, we study the hyperelliptic locus 
$\mathbb{M}^{\rm planar}_{g,{\rm hyp}}$ for arbitrary $g$, and we compute it explicitly for $g=3$.
The full moduli space 
is then obtained as
\begin{equation}
\label{eq:allofgenus3}
\mathbb{M}^{\rm planar}_3 \quad =\quad \mathbb{M}_{T_4} \,\cup \,
\mathbb{M}^{\rm planar}_{3,{\rm hyp}}.
\end{equation}
Just like in classical algebraic geometry,  ${\rm dim}(\mathbb{M}_{T_4}) = 6$
and ${\rm dim}(\mathbb{M}^{\rm planar}_{3,{\rm hyp}}) = 5$.

\begin{figure}[h]
  \centering
  \begin{subfigure}[b]{0.17\textwidth}
    \includegraphics[scale=1]{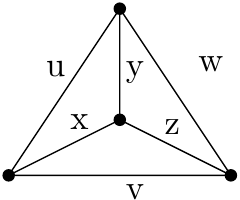}
    \caption*{(000)}
  \end{subfigure}
  \begin{subfigure}[b]{0.18\textwidth}
    \includegraphics[scale=0.92]{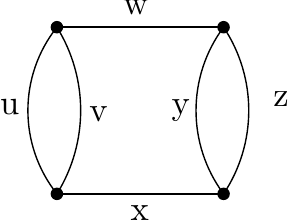}
    \caption*{(020)}
  \end{subfigure}
  \begin{subfigure}[b]{0.19\textwidth}
    \includegraphics[scale=1]{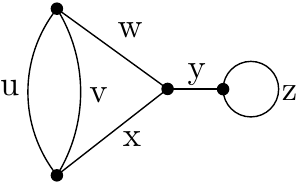}
    \caption*{(111)}
  \end{subfigure}
  \begin{subfigure}[b]{0.25\textwidth}
    \includegraphics[scale=1]{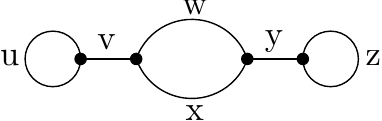}
    \caption*{(212)}
  \end{subfigure}
  \begin{subfigure}[b]{0.17\textwidth}
    \includegraphics[scale=1]{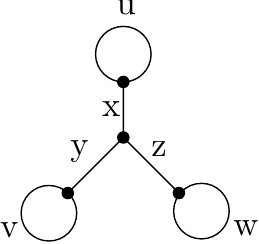}
    \caption*{(303)}
  \end{subfigure}
  \caption{The five trivalent  graphs of genus $3$, with letters labeling each graph's six edges}
  \label{figure:genus3_graphs}
\end{figure}

The stacky fan $\mathbb{M}_3$ of all metric graphs has five
maximal cones, as shown in \cite[Figure~4]{Chan}.  These correspond to the five (leafless) trivalent
graphs of genus 3, pictured in Figure \ref{figure:genus3_graphs}.  Each graph is labeled by the triple
$(\ell b c)$, where $\ell$ is the number of loops, $b$ is the number of bi-edges, and $c$ is the
number of cut edges. Here, $\ell$, $b$,  and $c$ are single digit
numbers, so there is no ambiguity to this notation.  Our labeling and ordering is largely
consistent with  \cite{Ba}.

Although $\mathbb{M}_{T_4}$ has dimension $6$, it is not pure due to the realizable metrics on (111). It also misses one of the five cones in
$\mathbb{M}_3$: the graph (303) cannot be realized in $\R^2$ 
by Proposition \ref{prop:sprawling}. The restriction of $\mathbb{M}_{T_4}$ to each of the other cones is
given by a finite union of convex polyhedral subcones, characterized by the following
piecewise-linear formulas:

\begin{theorem} 
  \label{thm:planequartics}
  A  graph in $\mathbb{M}_3$ arises from a smooth tropical quartic if and only if it is one of the
  first four  graphs in Figure \ref{figure:genus3_graphs}, 
  with edge lengths satisfying the following,
    up to symmetry:
  \begin{itemize}
\item  {\rm (000)} is realizable if and only if $\max\{x,y\} {\leq} u$, $\max\{x,z\} {\leq} v$
and $\max\{y,z\} {\leq} w$, where
\begin{itemize}
\vspace{-0.08in}
\item at most two of the inequalities can be equalities, and
\vspace{-0.05in}
\item if two are equalities, then either $x,y,z$ are distinct and 
  the edge (among $u,v,w$) that connects the shortest two of $x,y,z$ attains equality,  or
  $\max\{x,y,z\}$ is attained exactly twice, and the edge connecting those two longest does not attain
  equality.
\end{itemize}
\item {\rm (020)} is realizable if and only if $v\leq u$, $y\leq z$, and $ w+\max\{v,y\}\leq x$,
and if  the last inequality is an equality,
 then: $v=u$ implies $v<y<z$, and    $y=z$ implies $y<v<u$.
 \item {\rm (111)} is realizable if and only if $\,w<x\,$ and 
 \vspace{-0.1in}
\begin{equation}
\label{eq:logiclogic}
\!\!\!\!
\begin{matrix}
\!\! \hbox{$(\,v+w = x$ and $v<u\,)$}  \hbox{ or } 
\hbox{ $(\,v+w< x\leq v+3w$ and $v\leq u\,)$ } \hbox{ or } \\
 \hbox{ $(\,v+3w < x\leq v+4w$ and $v\leq u\leq 3v/2\,)$ }  \hbox{ or }  \\ \!\!
  \hbox{$(\,v+3w < x\leq v+4w$ and $2v=u\,)$} 
   \hbox{ or } 
\hbox{ $(\,v+4w< x\leq v+5w$ and $v=u\,)$.}
\end{matrix}
\end{equation}
\vspace{-0.3in}
\item {\rm (212)} is realizable if and only if $w<x\leq 2w$.
\end{itemize}
\end{theorem}

To understand the qualifier ``up to symmetry'' in Theorem~\ref{thm:planequartics},
it is worthwhile to read off the
automorphisms from the graphs in Figure~\ref{figure:genus3_graphs}.  The graph (000) is the
complete graph on four nodes.  Its automorphism group is the symmetric group of order
$24$.  The automorphism group of the graph (020) is generated by the three transpositions $(u\ v)$,
$(y\ z)$, $(w\ x)$ and the double transposition $(u\ y)(v\ z)$.  Its order is 16.  The automorphism
group of the graph (111) has order 4, and it is generated by $(u\ v)$ and $(w\ x)$.  The
automorphism group of the graph (212) is generated by $(u\ z)(v\ y)$ and $(w\ x)$, and has order $4$.
The automorphism group of the graph (303) is the symmetric group of order $6$.
Each of the five graphs contributes an orthant $\mathbb{R}_{\geq 0}^6$ modulo 
the action of that symmetry group to the stacky fan $\mathbb{M}_3$.

\begin{table}[h]
  \caption{Dimensions of the $1278$ moduli cones $\mathbb{M}_\Delta$ within $\mathbb{M}_{T_4}$}
  \label{tab:moduli:g3}
  \centering
  \begin{tabular*}{.66\linewidth}{@{\extracolsep{\fill}}lrrrrr@{}}
    \toprule
    $G$ $\backslash$ dim & 3 & 4 & 5 & 6 &  $\#\Delta\text{'s}$ \\
    \midrule
    (000) & 18 & 142 & 269 & 144 &  573\\
    (020) &    &  59 & 216 & 175 &  450 \\
    (111) &    &  10 & 120 &  95 &  225 \\
    (212) &    &     &  15 &  15 &   30 \\
    \midrule
    total   & 18 & 211 & 620 & 429 & 1278 \\
    \bottomrule
  \end{tabular*}
\end{table}

\begin{proof}[Proof of Theorem \ref{thm:planequartics}]
 This is based on explicit computations as in  Section~\ref{sec:combinat}. 
   The symmetric group $S_3$ acts on  the triangle $T_4$.  
   We enumerated all unimodular triangulations of $T_4$ up to that symmetry.
There are $1279$ (classes of) such triangulations, and of these precisely
  $1278$ are regular.  The unique non-regular triangulation is a refinement of \cite[Figure 2.3.9]{MS}.  For each regular triangulation we computed the graph $G$ and the polyhedral cone
  $\mathbb{M}_\Delta$.  Each $\mathbb{M}_\Delta$ is the image of the $12$-dimensional secondary cone
  of $\Delta$.  We found that $\mathbb{M}_\Delta$ has dimension $3$, $4$, $5$ or~$6$, depending on
  the structure of the triangulation $\Delta$.  A census is given by Table~\ref{tab:moduli:g3}.
For instance, $450$ of the $1278$ triangulations $\Delta$ have the skeleton $G = {\rm (020)}$.
Among these $450$, we found that 
$59$ have ${\rm dim}(\mathbb{M}_\Delta) = 4$,
$\,216$ have ${\rm dim}(\mathbb{M}_\Delta) = 5$, and
$175$ have ${\rm dim}(\mathbb{M}_\Delta) = 6$.

For each of the $1278$ regular triangulations $\Delta$ we  checked that the inequalities stated in
Theorem~\ref{thm:planequartics} are valid on the cone $\,\mathbb{M}_\Delta = 
(\kappa \circ \lambda)(\Sigma(\Delta))$.
 This proves that the dense realizable part of  $\mathbb{M}_{T_4}$ is contained
 in the polyhedral space described by our constraints. 

For the converse direction, we need to go through the four cases
and construct a planar tropical realization of each metric graph
that satisfies our constraints. We shall now do this.

\begin{figure}[h]
\centering
\includegraphics[scale=0.7]{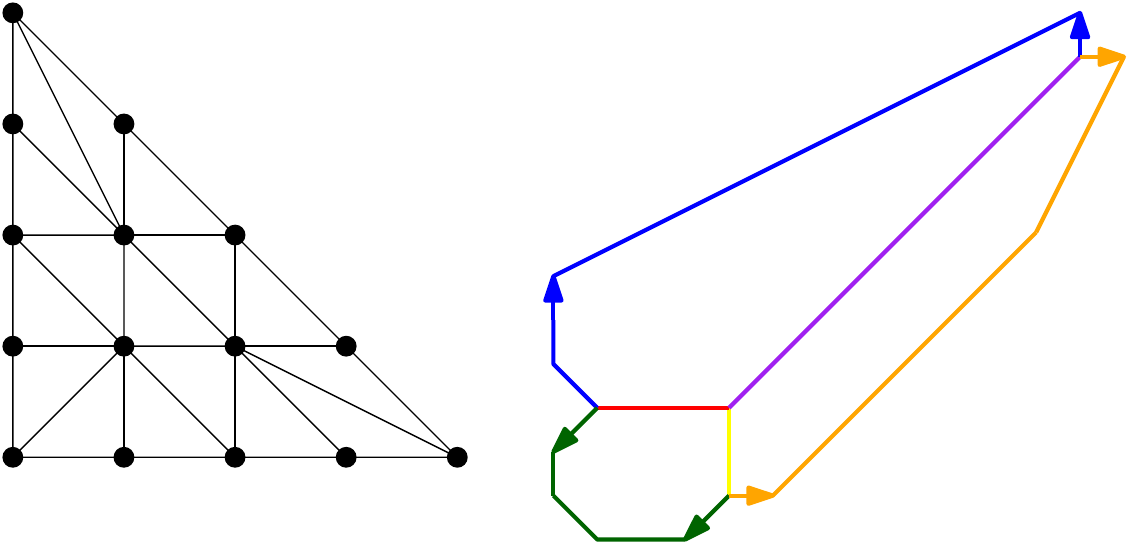}
\vspace{-0.05in}
\caption{A triangulation that realizes almost all realizable graphs of type (000)}
\label{figure:type1_generous}
\end{figure}

All realizable graphs of {\bf type (000)}, except for lower-dimensional
  families, arise from a single triangulation $\Delta$, shown in Figure
  \ref{figure:type1_generous}  with its skeleton. The cone $\mathbb{M}_\Delta$ is
  six-dimensional. Its interior is defined by
  $x <\min\{u, v\}$, $y < \min\{u, w\}$,  and $z<  \min\{v, w\}$.
Indeed, 
  the parallel segments in the outer edges can be arbitrarily long, and each outer edge be as close as desired to the maximum of the two adjacent inner edges.  This is accomplished by putting as
  much length as possible into a particular edge
  and pulling extraneous parts back.


\begin{figure}[h]
\centering
\includegraphics[scale=0.62]{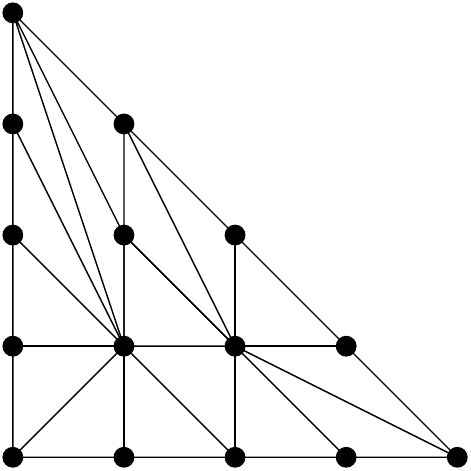} \,
\includegraphics[scale=0.62]{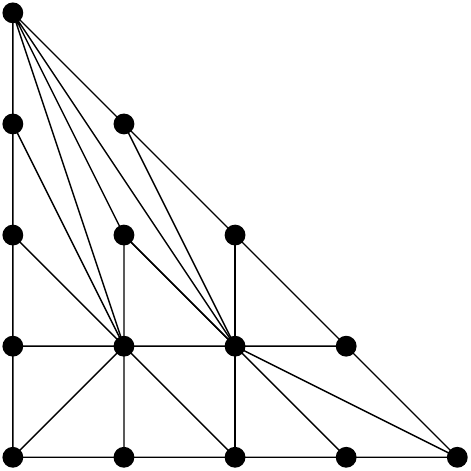} \,
\includegraphics[scale=0.62]{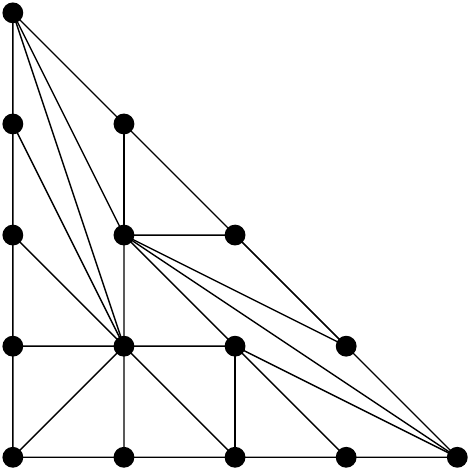} \,
\includegraphics[scale=0.62]{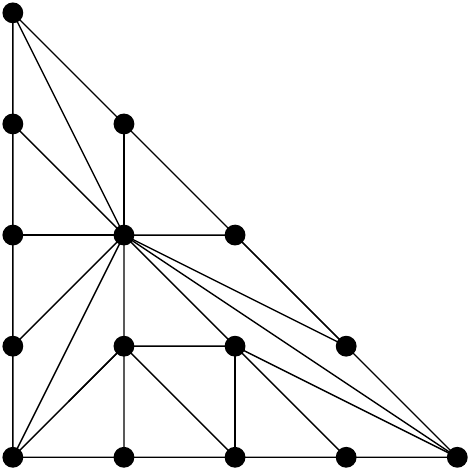} \,
\includegraphics[scale=0.62]{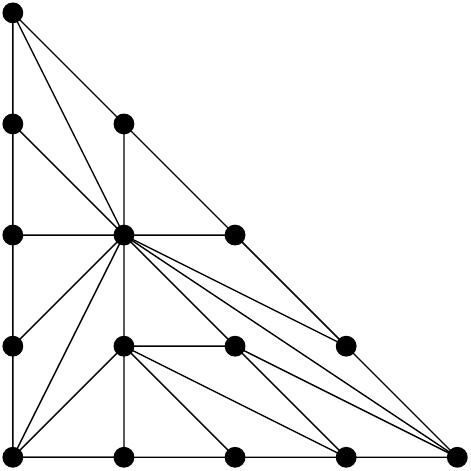}
\caption{Triangulations giving all metrics in the cases (i) through (v) for the graph (000)}
\label{figure:(000)3_boundary}
\end{figure}

There are several lower dimensional collections of graphs we must show are achievable:
\begin{itemize}
\vspace{-0.1in}
\item[(i)] $y<x=u$, $\max\{x,z\}<v$, $\max\{y,z\}<w$; \hfill  ($\dim=5$) 
\vspace{-0.15in}
\item[(ii)]  $y=x=u$, $\max\{x,z\}<v$, $\max\{y,z\}<w$; \hfill  ($\dim=4$) 
\vspace{-0.15in}
\item[(iii)]   $z< y<x<v$, $u=x$, $w=y$; \hfill  ($\dim=4$) 
\vspace{-0.15in}
\item[(iv)]  $z< y<x<u$, $v=x$, $w=y$; \hfill  ($\dim=4$) 
\vspace{-0.15in}
\item[(v)]  $z<y=x=v=w<u$. \hfill  ($\dim=3$)\,\,
\end{itemize}
In Figure \ref{figure:(000)3_boundary} we show triangulations 
realizing these five special families.  
Dual edges are labeled
$$(1,1)\buildrel{x}\over{\--}(1,2)\buildrel{y}\over{\--}(2,1)\buildrel{z}\over{\--}(1,1).$$





\begin{figure}[h]
\centering
\includegraphics[scale=0.7]{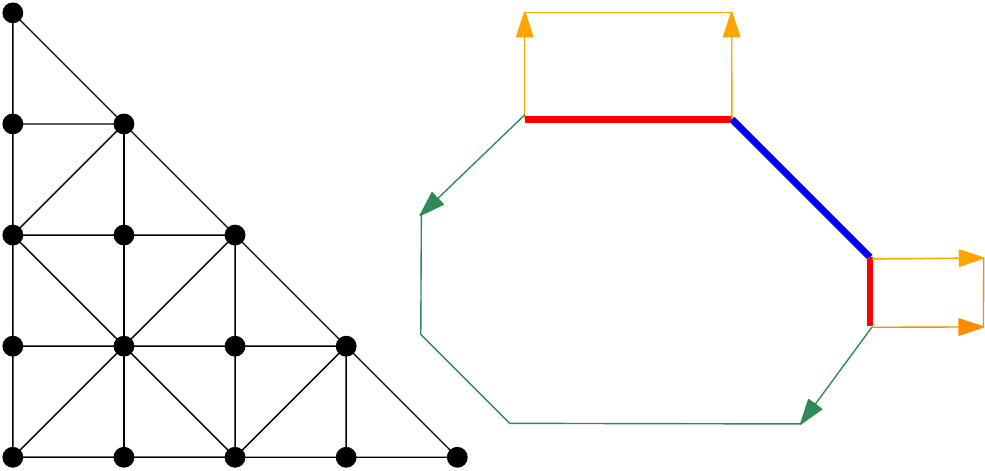}
\vspace{-0.05in}
\caption{A triangulation that realizes almost all realizable graphs of type (020)}
\label{figure:realizable_type_2}
\end{figure}

Next, we consider {\bf type (020)}. Again, except for some lower-dimensional cases, all  graphs
   arise from single triangulation, pictured in Figure \ref{figure:realizable_type_2}.  The interior of $\mathbb{M}_\Delta$ is given by $v<u$, $y<z$, and $w+\max\{v,y\}< x$.
There are several remaining boundary cases, all of whose graphs are realized by the triangulations in Figure \ref{figure:(020)3_boundary}:
\begin{itemize}
\vspace{-0.1in}
\item[(i)]  $v<u$, $y<z$, $w+\max\{v,y\}=x$; \hfill  ($\dim=5$) 
\vspace{-0.15in}
\item[(ii)]  $u=v$, $y<z$,  $w+\max\{v,y\}<x$; \hfill  ($\dim=5$) 
\vspace{-0.15in}
\item[(iii)] $u=v$, $y=z$,  $w+\max\{v,y\}<x$; \hfill  ($\dim=4$) 
\vspace{-0.15in}
\item[(iv)] $u=v$, $v<y<z$,  $w+\max\{v,y\}=x$. \hfill  ($\dim=4$)\,\, 
\end{itemize}

\begin{figure}[h]
\centering
\includegraphics[scale=0.64]{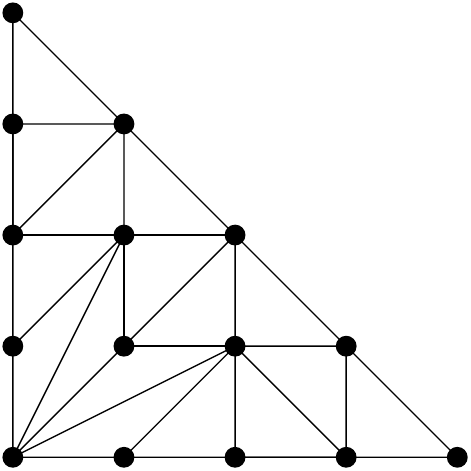} \,
\includegraphics[scale=0.64]{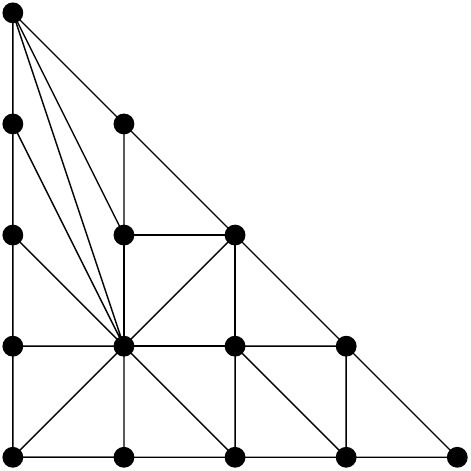} \,
\includegraphics[scale=0.64]{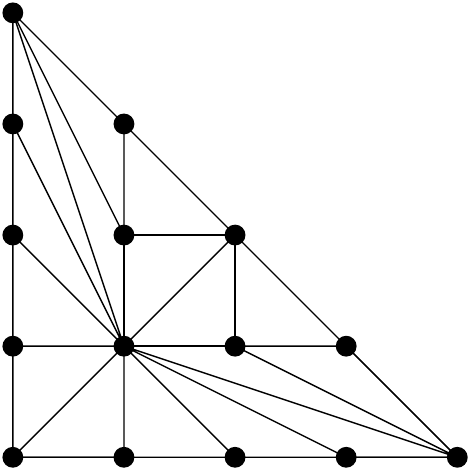}\,
\includegraphics[scale=0.64]{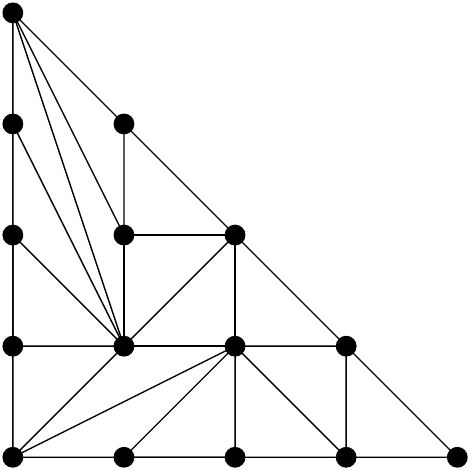}
\caption{Triangulations giving all metrics in the cases (i) through (iv) for the graph (020)}
\label{figure:(020)3_boundary}
\end{figure}

{\bf Type (111)} is the most complicated.  We begin by realizing the
metric graphs that lie in $\text{int}(\mathbb{M}_{T_4,(111)})$.  These arise from 
the second and third cases  in the disjunction (\ref{eq:logiclogic}).

We assume $w<x$.  The  triangulation to the left in Figure \ref{figure:type3_case1_case2} realizes all metrics on (111) satisfying  $v+w< x< v+3w$ and $v<u$.
The dilation freedom of $u$, $y$, and $z$ is clear.  To see that the  edge $x$ can have length arbitrarily close to $v+3w$, simply dilate the double-arrowed segment to be as long as possible, with some very small length given to the next two segments counterclockwise.   Shrinking the double-arrowed segment as well as the vertical segment of $x$ brings the length close to $v+w$.  The
triangulation to the right in Figure \ref{figure:type3_case1_case2} realizes all metrics satisfying  $v+3w < x< v+4w$ and $v<u<3v/2$. Dilation of $x$ is more free due to the double-arrowed segment of slope $1/2$, while dilation of $u$ is more restricted.

\begin{figure}[h]
\centering
\includegraphics[scale=0.6]{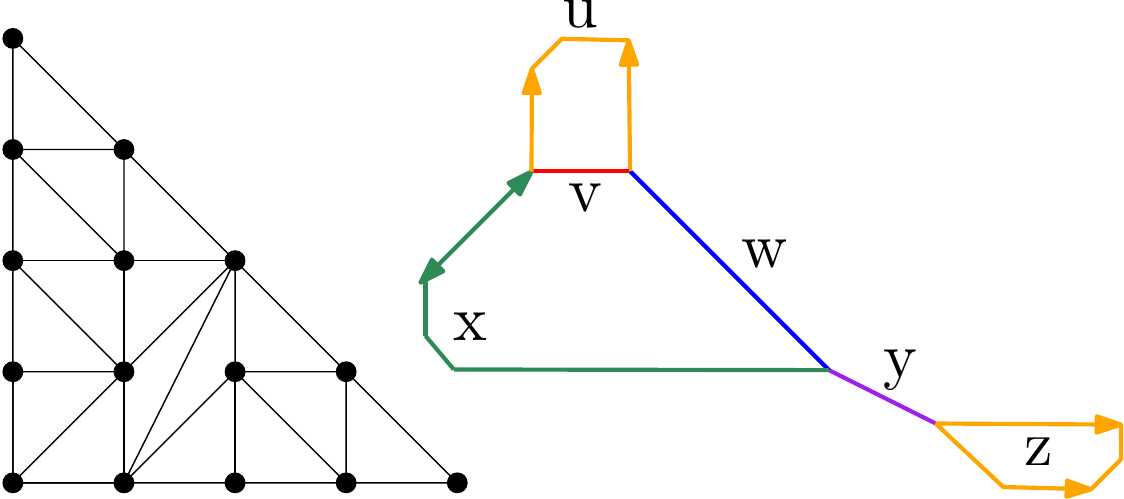}\qquad
\includegraphics[scale=0.6]{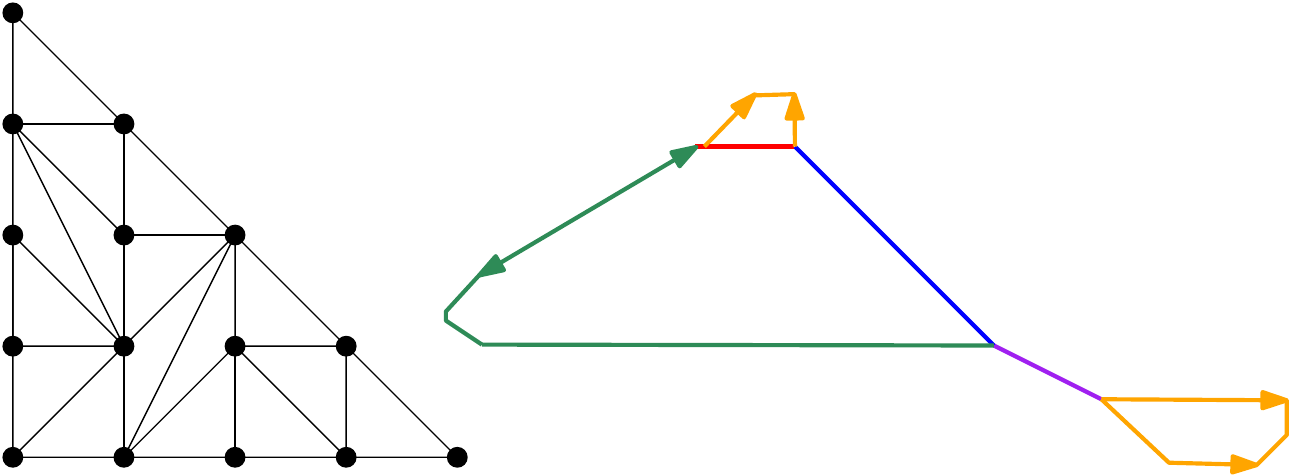}
\caption{Triangulations of type (111) realizing $v+w< x< v+2x$ and $v<u$ (on the left) and $v+3w < x< v+4w$ and {$v<u<3v/2$} (on the right)}
\label{figure:type3_case1_case2}
\end{figure}

Many triangulations are needed in order to deal with low-dimensional case.  In Figure \ref{figure:(111)3_boundary} we show triangulations that realize each of the following families of type (111) graphs:
\begin{itemize}
\vspace{-0.1in}
\item[(i)] $v+w<x<v+5w$, $v=u$;   \hfill  ($\dim=5$) 
\vspace{-0.13in}
\item[(ii)] $v+w<x<v+4w$, $2v=u$;   \hfill  ($\dim=5$) 
\vspace{-0.13in}
\item[(iii)]  $v+w=x$, $v<u$;   \hfill  ($\dim=5$) 
\vspace{-0.13in}
\item[(iv)] $x=v+3w$, $v<u$;   \hfill  ($\dim=5$) 
\vspace{-0.13in}
\item[(v)] $x=v+4w$, $v<u\leq 3v/2$;   \hfill  ($\dim=5$) 
\vspace{-0.13in}
\item[(vi)] $x=v+5w$, $v=u$;   \hfill  ($\dim=4$)
\vspace{-0.13in}
\item[(vii)] $x=v+4w$, $2v=u$.   \hfill  ($\dim=4$)\,\,
\end{itemize}

\begin{figure}[h]
\centering
\includegraphics[scale=0.42]{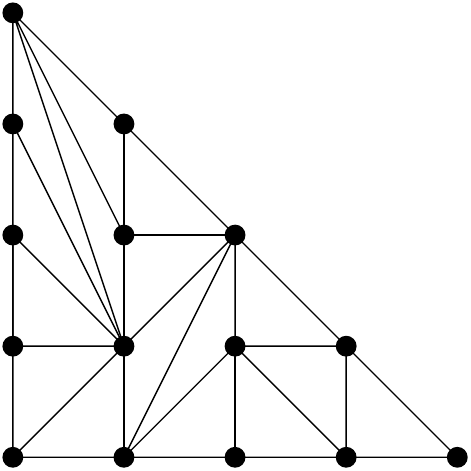} \,
\includegraphics[scale=0.42]{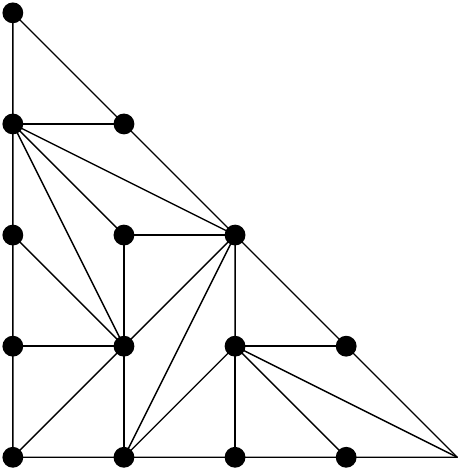} \,
\includegraphics[scale=0.42]{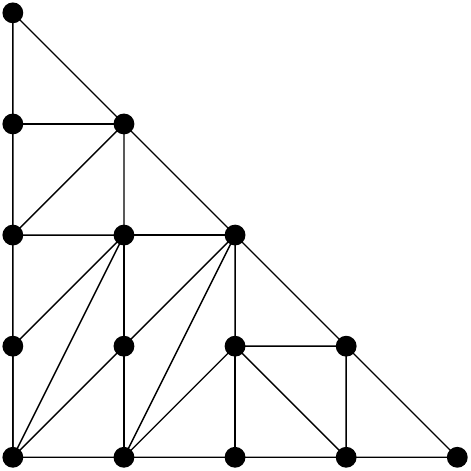} \,
\includegraphics[scale=0.42]{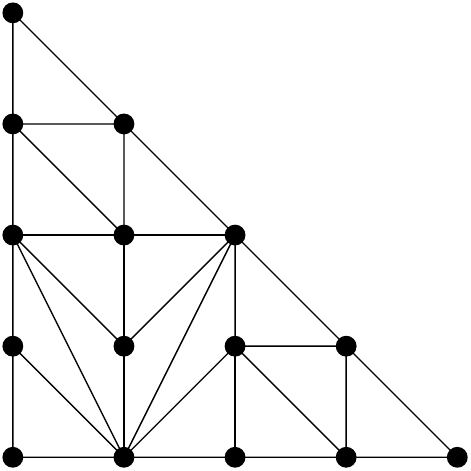} \,
\includegraphics[scale=0.42]{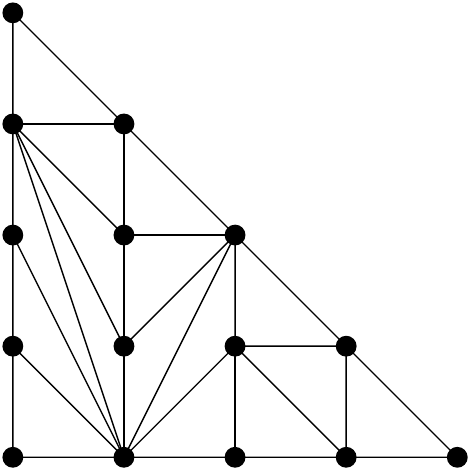} \,
\includegraphics[scale=0.42]{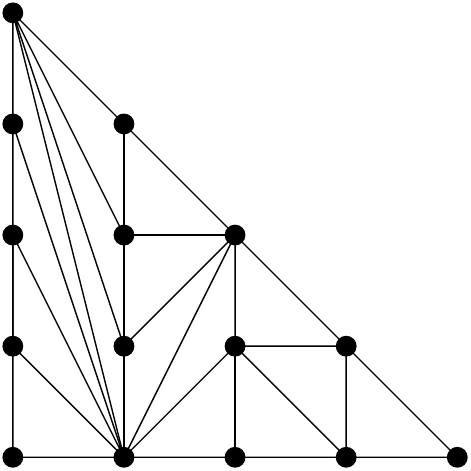} \,
\includegraphics[scale=0.42]{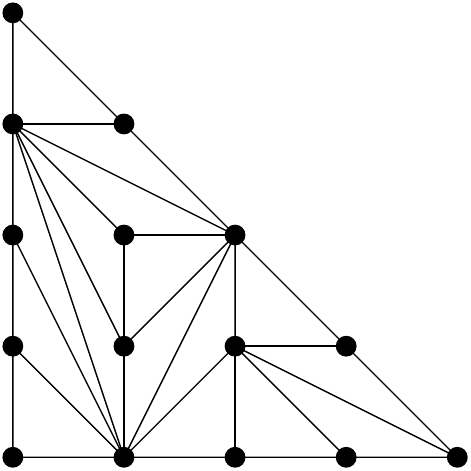}
\caption{Triangulations of type (111) that realize
 the boundary cases (i) through (vii)}
\label{figure:(111)3_boundary}
\end{figure}

 All graphs of {\bf type (212)} can be achieved with the two triangulations in Figure \ref{figure:type4_generous_and_boundary}.  The left gives all possibilities with $w<x< 2w$, and the right realizes $x=2w$.  The edges $u$, $v$, $y$, $z$ are completely
  free to dilate. 
 This completes the proof of Theorem \ref{thm:planequartics}.
\end{proof}

\begin{figure}[h]
\centering
\includegraphics[scale=0.65]{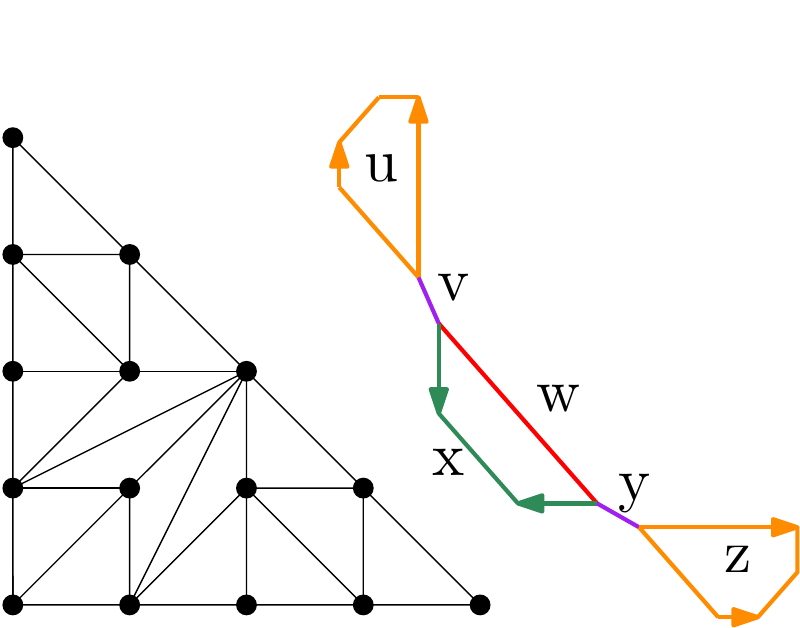} \qquad \qquad
\includegraphics[scale=0.65]{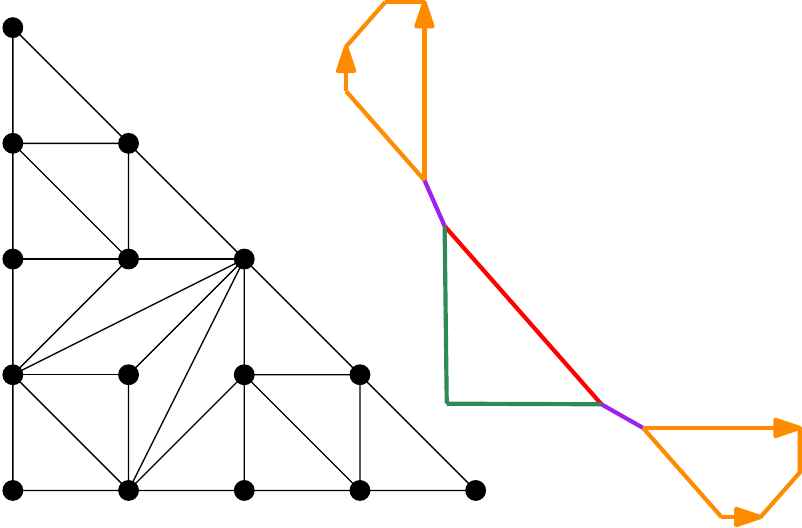}
\caption{Triangulations giving graphs of type (212) giving $w<x< 2w$ and $x=2w$}
\label{figure:type4_generous_and_boundary}
\end{figure}

The space $\mathbb{M}_{T_4}$ is not pure-dimensional 
because of the graphs  (111) with $u=v$ and $v+4w<x<v+5w$.  These  appear in the five-dimensional $\mathbb{M}_\Delta$ where $\Delta$ is the leftmost triangulation in Figure \ref{figure:(111)3_boundary}, but $\mathbb{M}_{\Delta}$ is not contained in the boundary of any six-dimensional~$\mathbb{M}_{\Delta'}$.

We close this section by suggesting an answer to the following question: \emph{What is the
  probability that a random metric graph of genus $3$ can be realized by a tropical plane quartic?}

To examine this question, we need to endow the moduli space $\mathbb{M}_3$ with a probability
measure.  Here we fix this measure as follows. We assume that the five trivalent graphs $G$ are equally
likely, and all non-trivalent graphs have probability $0$.  The lengths on each trivalent graph $G$
specify an orthant $\R^6_{\geq  0}$. We fix a probability measure
on $\R^6_{\geq  0}$  by normalizing so that $u+v+w +x+y+z= 1$,
 and we take the
uniform distribution on the resulting $5$-simplex.  With this probability
measure on the moduli space $\mathbb{M}_3$ we are asking for the ratio of
volumes
\begin{equation}
\label{eq:probability} 
\vol(\mathbb{M}^{\rm planar}_3)/
\vol(\mathbb{M}_3).
\end{equation}
This ratio is a rational number, which we computed from our data in
Theorem \ref{thm:planequartics}.

\begin{corollary}\label{cor:g3:probability}
  The rational number in (\ref{eq:probability}) is $31/105$.  This means that, in the measure specified above, about
  $29.5$\% of all metric graphs of genus $3$ come from tropical plane quartics.
\end{corollary}

\begin{proof}[Proof and Explanation]
  The graph (303) is not realizable, since none of the $1278$ regular unimodular triangulations of the triangle $T_4$
  has this type. So, its probability is zero.  For the other four trivalent graphs in Figure \ref{figure:genus3_graphs}
  we compute the volume of the realizable edge lengths, using the inequalities in Theorem \ref{thm:planequartics}. The
  result of our computations is the table
\begin{center}
  \begin{tabular*}{.67\linewidth}{@{\extracolsep{\fill}}lrrrrr@{}}
    \toprule
    Graph       & (000) & (020) &  (111) & (212) & (303) \\
    Probability &  4/15 &  8/15 & 12/35 &   1/3 &     0 \\ 
    \bottomrule
  \end{tabular*}
\end{center}

A non-trivial point in verifying these numbers is that Theorem~\ref{thm:planequartics} gives the constraints only up to
symmetry. We must apply the automorphism group of each graph in order to obtain the realizable region in its $5$-simplex
$\{(u,v,w,x,y,z) \in \mathbb{R}^6_{\geq 0} : u+v+w+x+y+z = 1 \}$.  Since we are measuring volumes, we are here allowed
to replace the regions described in Theorem~\ref{thm:planequartics} by their closures.  For instance, consider type
(020). After taking the closure, and after applying the automorphism group of order $16$, the realizability condition
becomes
\begin{equation}
\label{eq:maxmin}
 \max \bigl( \min(u,v), \min(y,z) \bigr) \,\, \leq \,\, | x-w|. 
 \end{equation}
 The probability that a uniformly sampled random point in the $5$-simplex satisfies (\ref{eq:maxmin}) is equal to
 $8/15$.  The desired probability (\ref{eq:probability}) is the average of the five numbers in the table.
\end{proof}

Notice that asking for those probabilities only makes sense since the dimension of the moduli space agrees with the
number of skeleton edges.  In view of (\ref{eq:dimformula}) this occurs for the three genera $g=2,3,4$.  For $g\ge 5$
the number of skeleton edges exceeds the dimension of the moduli space.  Hence, in this case, the probability that a
random metric graph can be realized by a tropical plane curve vanishes a priori.  For $g=2$ that probability is one; see
Example~\ref{eq:genus2}.  For $g=4$ that probability is less than $0.5\%$ by Corollary~\ref{cor:g4:probability} below.

\section{Hyperelliptic Curves}
\label{sec:hyperelliptic}

A polygon $P$ of genus $g$ is \emph{hyperelliptic}
if  $P_{\interior}$ is a line segment of length $g-1$.
We define the moduli space of hyperelliptic tropical plane curves of genus $g$ to be
$$ \mathbb{M}^{\rm planar}_{g,{\rm hyp}}  \,\, \, := \,\,\,
\bigcup_P \mathbb{M}_P, $$
where the union is over all hyperelliptic polygons $P$ of genus $g$.
Unlike when the interior hull $P_{\rm{int}}$ is two-dimensional, there does not exist a unique maximal 
hyperelliptic polygon $P$ with given $P_{\rm{int}}$.
   However, there are only finitely many  
such polygons  up to isomorphism.~These~are
~$$E_k^{(g)} \,\,:= \,\, \text{conv}\{(0,0),(0,2),(g+k,0),(g+2-k,2)\}
\qquad \hbox{for}\,\,\, 1\leq k\leq g+2 .$$
These hyperelliptic polygons interpolate between the
   rectangle $E_1^{(g)} = R_{g+1,2}$ and the triangle $E_{g+2}^{(g)}$. 
  The five maximal hyperelliptic polygons for  genus $g = 3$  are pictured in Figure~\ref{figure:g3hyperelliptic_polygons}.

\begin{figure}[h]
\centering
\includegraphics[scale=0.8]{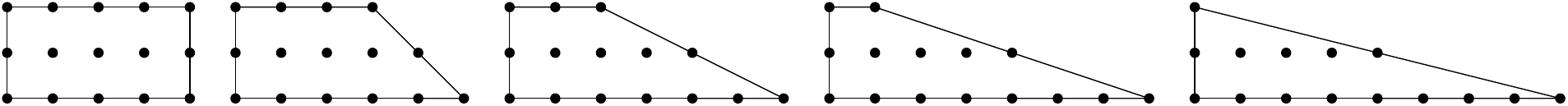}
\caption{The five maximal hyperelliptic polygons of genus $3$}
\label{figure:g3hyperelliptic_polygons}
\end{figure}

This finiteness property makes a computation of $\mathbb{M}^{\rm planar}_{g, {\rm hyp}}$
feasible:  compute  $\mathbb{M}_{E_k^{(g)}}$ for all $k$, 
and take the union. By \cite[Proposition 3.4]{KZ}, all triangulations of hyperelliptic polygons are regular, so we need not worry about non-regular triangulations arising in the \topcom computations described in Section \ref{sec:combinat}.
We next show that it suffices to consider the triangle:

\begin{theorem}
\label{thm:chains}
For each genus $g \geq 2$, the hyperelliptic triangle $E^{(g)}_{g+2}$ satisfies
\begin{equation}
\label{eq:hyptriangle}
\mathbb{M}_{E^{(g)}_{g+2}}
\,\,=\,\,\mathbb{M}^{\rm planar}_{g,{\rm hyp}} \,\, \subseteq \,\,
\mathbb{M}_g^{\rm chain}\cap \mathbb{M}_g^{\rm planar}.
\end{equation}
The equality holds even before taking closures of the spaces of realizable graphs.
The spaces on the left-hand side and right-hand side of
the inclusion in (\ref{eq:hyptriangle}) both have dimension $2g-1$.
\end{theorem}

Before proving our theorem, we define $\mathbb{M}_g^{\rm chain}$.  This 
space contains all metric graphs that arise from triangulating hyperelliptic polygons.
Start with a line segment on $g-1$ nodes where the $g-2$ edges have 
arbitrary non-negative lengths. Double each edge so that
the resulting parallel edges have the same length, and attach two loops of arbitrary lengths
at the endpoints.
Now, each of the $g-1$ nodes is $4$-valent.
There are two possible ways to split each node into two nodes connected
by an edge of arbitrary length.
Any metric graph arising from this procedure is called a \emph{chain of genus $g$}.  Although there are $2^{g-1}$ possible choices in this procedure, some give isomorphic graphs.  There are $2^{g-2}+2^{\lfloor(g-2)/2\rfloor}$ combinatorial types of chains of genus $g$.  In genus $3$ the chains are (020), (111), and (212)  in Figure \ref{figure:genus3_graphs}, and in  genus $4$ they
are  (020), (021), (111), (122), (202), and (223)  in Figure \ref{figure:genus4_graphs}.

By construction, there are $2g-1$ degrees
of freedom for the edge lengths in
a chain of genus $g$, so each such chain defines an orthant
$\mathbb{R}_{\geq 0}^{2g-1}$.
We write $\mathbb{M}^{\rm chain}_g$ for the
stacky subfan of $\mathbb{M}_g$ consisting of all chains.
Note that $\mathbb{M}_g^{\rm chain}$
is strictly contained in the space $\mathbb{M}^{\rm hyp}_g$ of all
hyperelliptic metric graphs,  seen in \cite{Chan2}.
 Hyperelliptic graphs arise
by the same construction from any tree with $g-1$ nodes, 
whereas for chains that tree must be a line segment.

The main claim in Theorem \ref{thm:chains} is that any metric graph arising from a maximal hyperelliptic polygon $E_k^{(g)}$ also arises from the hyperelliptic triangle $E^{(g)}_{g+2}$.  Given a triangulation $\Delta$ of $E_k^{(g)}$, our proof constructs a triangulation $\Delta'$ of $E^{(g)}_{g+2}$ that gives rise to the same collection of metric graphs, so that $\mathbb{M}_{\Delta}=\mathbb{M}_{\Delta'}$, with equality holding even before taking closures.  Before our proof, we illustrate this construction with the following example.

 \begin{example}\label{example:tau} \rm 
 
   Let $\Delta$ be the triangulation of $R_{4,2}$ pictured on the left in Figure \ref{figure:tau_input} along with a metric graph $\Gamma$ arising from it.  The possible metrics on $\Gamma$ are determined by the slopes of the edges emanating from the vertical edges.  For instance, consider the constraints on $v$ and $y$ imposed by the width $w$ (which equals $x$).  If most of the $w$ and $x$ edges are made up of the segments emanating from $v$, we find $y$ close to $v+2w$.  If instead most of the $w$ and $x$ edges are made up of the segments emanating from $y$, we find $y$ close to $v-2w$.  Interpolating gives graphs achieving $v-2w<y<v+2w$.  This only depends on the \emph{difference} of the slopes emanating either left or right from the edges $v$ and $y$: the same constraints would be imposed if the slopes emanating from $v$ to the right were $2$ and $0$ rather than $1$ and $-1$.  Boundary behavior determines constraints on $u$ and $z$, namely $v<u$ and $y<z$.
   
      \begin{figure}[h]
\centering
\includegraphics[scale=0.84]{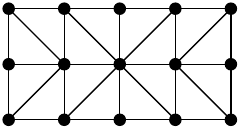}\qquad
\includegraphics[scale=0.55]{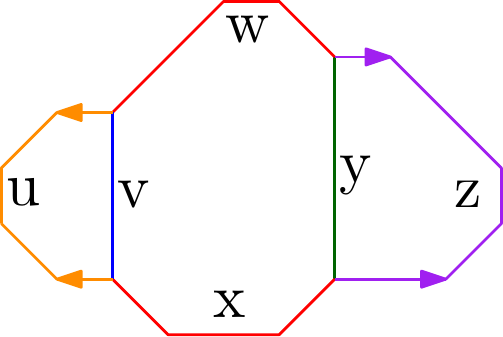} \qquad \quad
\includegraphics[scale=0.78]{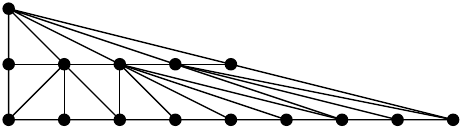} \qquad
\includegraphics[scale=0.45]{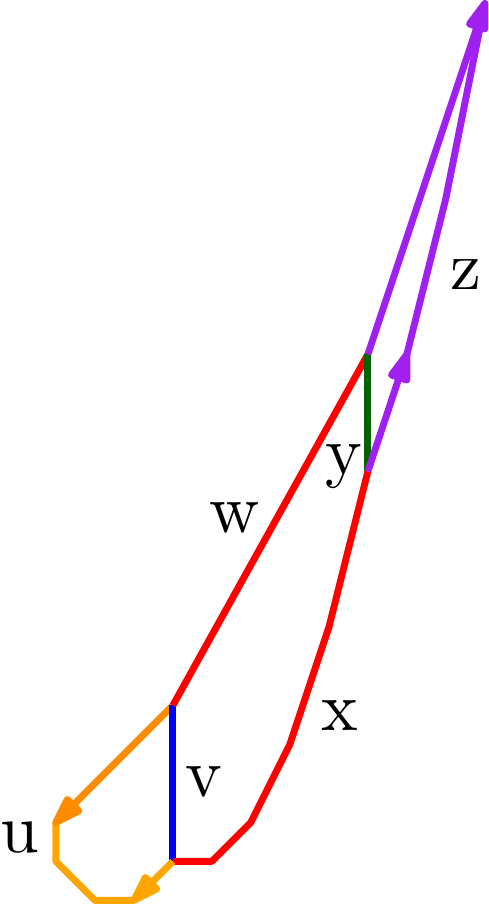}
\caption{Triangulations of $R_{4,2}$ and $E^{(3)}_{5}$, giving rise to skeletons with the same  metrics.}
\vspace{-0.1in}
\label{figure:tau_input}
\end{figure}
 
Also pictured in Figure \ref{figure:tau_input} is a triangulation $\Delta'$ of $E^{(3)}_{5}$.  The skeleton $\Gamma'$ arising from $\Delta'$ has the same combinatorial type as $\Gamma$, and the slopes emanating from the vertical edges have the same differences as in $\Gamma$.  Combined with similar boundary behavior, this shows that $\Gamma$ and $\Gamma'$ have the exact same achievable metrics.  In other words, $\mathbb{M}_{\Delta}=\mathbb{M}_{\Delta'}$, with equality  even before taking closures of the realizable graphs.

 We now explain how to construct $\Delta'$ from $\Delta$, an algorithm spelled out explicitly for general $g$ in the proof of Theorem \ref{thm:chains}.  We start by adding edges from $(0,2)$ to the interior lattice points (since any unimodular triangulation of $E_5^{(3)}$ must include these edges), and then add additional edges based on the combinatorial type of $\Delta$, as pictured in Figure \ref{figure:tau_part_1}.
 
   \begin{figure}[h]
\centering
\includegraphics[scale=0.84]{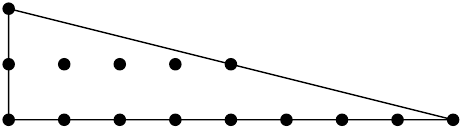} \quad
\includegraphics[scale=0.84]{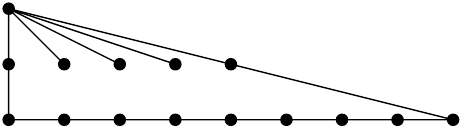} \quad
\includegraphics[scale=0.84]{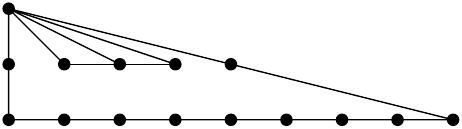} \quad
\caption{The start of $\Delta'$.}
\vspace{-0.1in}
\label{figure:tau_part_1}
\end{figure}

 Next we add edges connecting the interior lattice points to the lower edge of the triangle.  We will ensure that the outgoing slopes from the vertical edges in the $\Gamma'$  have the same difference as in $\Gamma$.  For $i=1,2,3$, we connect $(i,1)$ to all points between $(2i+a_i,0)$ and $(2i+b_i,0)$ where $a_i$ is the difference between the reciprocals of the slopes of the leftmost edges from $(i,1)$ to the upper and lower edges of $R_{4,2}$ in $\Delta$, and $b_i$ is defined similarly but with the rightmost edges.  Here we take the reciprocal of $\infty$ to be $0$.  In the dual tropical curve, this translates to slopes emanating from vertical edges in the tropical curve having the same difference as from $\Delta$.
 
    \begin{figure}[h]
\centering
\includegraphics[scale=0.78]{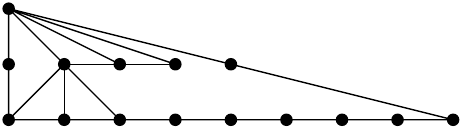}\quad
\includegraphics[scale=0.78]{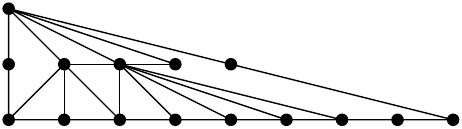} \quad
\includegraphics[scale=0.78]{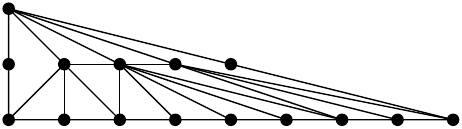} \quad
\includegraphics[scale=0.78]{tau_step6} \quad
\caption{Several steps leading up to $\Delta'$, on the right.}
\label{figure:tau_part_2}
\end{figure}
 
 We compute $a_1=\frac{1}{-1}-\frac{1}{1}=-2$ and $b_1=\frac{1}{\infty}-\frac{1}{\infty}=0$. Since 
 $2\cdot 1+a_1=0$ and $2\cdot 1+b_1=2$, we add edges from $(1,1)$ to $(0,0)$, to $(0,2)$, 
 and to all points in between, in this case just $(0,1)$.  We do similarly for the other two 
 interior lattice points, as pictured in the first three triangles in Figure \ref{figure:tau_part_2}.  The fourth triangle includes the edges $(0,1)-(1,1)$ and $(3,1)-(4,1)$, which ensures the same constraints as from $\Delta$ on the first and third loops of the corresponding metric graph.
 
    \begin{figure}[h]
\centering
\includegraphics[scale=0.78]{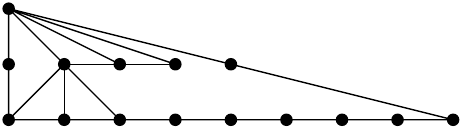}\quad
\includegraphics[scale=0.78]{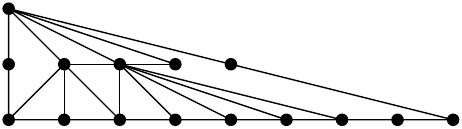} \quad
\includegraphics[scale=0.78]{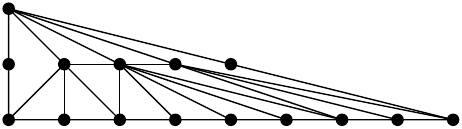} \quad
\includegraphics[scale=0.78]{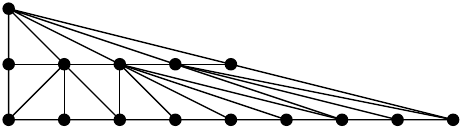} \quad
\caption{Several steps leading up to $\Delta'$, on the right.}
\label{figure:tau_part_2}
\end{figure}
  
 \end{example}

 \begin{proof}[Proof of Theorem \ref{thm:chains}]
 The inclusion $\mathbb{M}^{\rm planar}_{g,{\rm hyp}} \subseteq 
\mathbb{M}_g^{\rm chain}$ holds because every unimodular triangulation of
a hyperelliptic polygon is dual to a chain graph. Such a chain  has $2g-1$ edges, 
and hence ${\rm dim}(\mathbb{M}_g^{\rm chain}) = 2g-1$. We also have
${\rm dim}(\mathbb{M}^{\rm planar}_{g,{\rm hyp}}) \geq 2g-1$
because 
Lemma \ref{lem:honeydim} implies ${\rm dim}(\mathbb{M}_{R_{g+1,2}}) = 2g-1$. 
Hence the inclusion  implies the dimension statement.

It remains to prove the equality $\mathbb{M}_{E^{(g)}_{g+2}} = \mathbb{M}^{\rm planar}_{g,{\rm hyp}}$.
Given any triangulation $\Delta$ of a hyperelliptic polygon $E_k^{(g)}$, we shall construct a 
triangulation $\Delta'$ of $E_{g+2}^{(g)}$ such that $\mathbb{M}_{\Delta}=\mathbb{M}_{\Delta'}$.
Our construction will show that the equality even holds 
at the level of smooth tropical curves.

We start constructing $\Delta'$ by drawing
 $g$ edges from $(0,2)$ to the interior lattice points.  
 The next $g-1$ edges of $\Delta'$ are those that give it the same skeleton as
   $\Delta$.  This means that $\Delta'$ has the edge
   $(i,1)-(i+1,1)$ whenever that edge is in $\Delta$,
   and $\Delta'$ has the edge   $(0,2)-(2i+1,0)$ whenever $(i,1)-(i+1,1)$ is not
   an edge in $\Delta$. Here $i=1,\ldots,g-1$.
     
Next we will include edges in $\Delta'$ that give the same constraints on vertical edge lengths as $\Delta$.  This is accomplished by connecting the  point $(i,1)$ to  $(2i+a_i,0)$, to $(2i+b_i,0)$, and to all points in between, where $a_i$ and $b_i$ are defined as follows.  Let $a_i$ be the difference between the reciprocals of the slopes
 of the leftmost edges from $(i,1)$ 
to the upper and lower edges of $E_k^{(g)}$ in $\Delta$.  Here we take the reciprocal of $\infty$ to be $0$.  Let $b_i$ be defined similarly, but with the rightmost edges.  These new edges in $\Delta'$ do not cross due to constraints on the slopes in $\Delta$. Loop widths and differences in extremal slopes determine upper and lower bounds on the lengths of vertical edges.  These constraints  on the $g-2$ interior loops mostly guarantee $\mathbb{M}_{\Delta}=\mathbb{M}_{\Delta'}$.  To take care of the $1^{st}$ and $g^{th}$ loops, we must complete the definition of $\Delta'$.
Let $(n,0)$ be the leftmost point of the bottom edge of $E^{(g)}_{g+2}$ connected to $(1,1)$ so far in~$\Delta'$.
\begin{itemize}
\item[(i)] If $n=0$ then $\Delta'$ includes the edge  $(0,1)-(1,1)$.
\vspace{-0.1in}
\item[(ii)] If $n\geq 2$ then $\Delta'$ includes $(0,1)-(1,1)$ and all edges $(0,1)-(0,m)$ with $0\leq m\leq n$.  
\vspace{-0.1in}
\item[(iii)] If $n {=} 1$ and $(0,1)-(1,1)$ is an edge of $\Delta$ then $\Delta'$ includes $(0,1)-(1,1)$ and $(0,1)-(1,0)$.
\vspace{-0.3in}
\item[(iv)] If $n{=}1$ and $(0,1)-(1,1)$ is not an edge 
$\Delta$  then  $\Delta'$ includes $(0,2)-(1,0)$ and $(0,1)-(1,0)$.
\end{itemize} Perform a symmetric construction around $(g,1)$.  These edge choices will give the same constraints on the $1^{st}$ and $g^{th}$ loops as those imposed by $\Delta$.  This completes the proof.
\end{proof}

We now return to genus $g=3$, our topic in Section \ref{sec:genus3}, and we complete the computation of $\mathbb{M}^{\rm
  planar}_3$.
By \eqref{eq:allofgenus3} and Theorem \ref{thm:chains}, it suffices to compute the $5$-dimensional space
$\mathbb{M}_{E_{g+2}^{(g)}}$.  An explicit computation as in Section \ref{sec:combinat} reveals that the rectangle
$E^{(3)}_1 = R_{4,2}$ realizes precisely the same metric graphs as the triangle $E^{(3)}_5$.  With this, Theorem \ref
{thm:chains} implies $\mathbb{M}^{\rm planar}_{3,{\rm hyp}} = \mathbb{M}_{R_{4,2}}$.  To complete the computation in
Section \ref{sec:genus3}, it thus suffices to analyze the rectangle $R_{4,2}$.
  
\begin{table}[h]
  \caption{Dimensions of the moduli cones $\mathbb{M}_\Delta$ for $R_{4,2}$ and $E^{(3)}_{5}$}
  \label{tab:moduli:hypg3}
  \centering
  \begin{tabular*}{\linewidth}{@{\extracolsep{\fill}}l@{\hskip .6in}rrrr@{\hskip .6in}rrrr@{}}
    \toprule
    &&\multicolumn{2}{c}{$R_{4,2}$} &&& \multicolumn{2}{c}{$E^{(3)}_{5}$} \\
    \midrule
    $G$ $\backslash$ dim & 3 & 4 & 5 &  $\#\Delta\text{'s}$ \,\, & 3 & 4 & 5 &  $\#\Delta\text{'s}$ \\
    \midrule
    (020) &   42 &  734 & 1296 & 2072 \,\, & 42 & 352 & 369 & 763 \\
    (111) &    &  211 & 695 &   906 \,\, & & 90& 170 & 260   \\
    (212) &    &     &  127 &    127 \,\, & & & 25 & 25 \\
    \midrule
    total   & 42 & 945 & 2118 & 3105 \,\, & 42 & 442 & 564 & 1048  \\
    \bottomrule
  \end{tabular*}
\end{table}

It was proved in \cite{BLMPR} that $\mathbb{M}_{R_{4,2}}$ and $\mathbb{M}_{T_4}$ have disjoint interiors.  Moreover,
$\mathbb{M}_{R_{4,2}}$ is not contained in $\mathbb{M}_{T_4}$.  This highlights a crucial difference between
(\ref{eq:genus3classical}) and (\ref{eq:allofgenus3}).  The former concerns the tropicalization of classical moduli
spaces, so the hyperelliptic locus lies in the closure of the non-hyperelliptic locus.  The analogous statement is false
for tropical plane curves.  To see that $\mathbb{M}_{T_4}$ does not contain $\mathbb{M}_{R_{4,2}}$ consider the (020)
graph with all edge lengths equal to $1$.  By Theorems \ref{thm:planequartics} and \ref{thm:hypg3}, this metric graph is
in $\mathbb{M}_{R_{4,2}}$ but not in $\mathbb{M}_{T_4}$.  What follows is the hyperelliptic analogue to the
non-hyperelliptic Theorem \ref{thm:planequartics}.

\begin{theorem} 
  \label{thm:hypg3}
  A  graph in $\mathbb{M}_3$ arises from  $R_{4,2}$ if and only if it is one of the graphs (020), (111), or (212) in Figure \ref{figure:genus3_graphs}, 
  with edge lengths satisfying the following,
    up to symmetry:
  \begin{itemize}
    \item {\rm (020)} is realizable if and only if $w=x$, $v\leq u$, $v\leq y\leq z$, and 
  \vspace{-0.1in}
\begin{equation}
\label{eq:logiclogic1}
\!\!\!\!
\begin{matrix}
\!\! \hbox{$(y < v + 2w\,)$ }  \hbox{ or } 
\hbox{ $( y =v + 2w$ and $y<z\,)$ } \\ \hbox{ or } 
\hbox{ $( y < v+3w$ and $u\leq 2v\,)$ }  \hbox{ or }
\hbox{ $( y = v+3w$ and $u\leq 2v$ and $y< z\,)$ }\\ \hbox{ or } 
\hbox{ $(y < v+4w$ and $u=v\,)$ } 
 \hbox{ or } 
\hbox{ $(\,y = v +4w$ and $u=v$ and $y< z\,)$.}
\end{matrix}
\end{equation}
\vspace{-0.24in}
 \item   {\rm (111)} is realizable if and only if $w=x$ and $\min\{u,v\}\leq w$.
 \vspace{-0.08in}
\item {\rm (212)} is realizable if and only if $w=x$.
\end{itemize}
\end{theorem}

\begin{proof}
 This is based on an explicit computation as described in
  Section~\ref{sec:combinat}. The hyperelliptic rectangle $R_{4,2}$ has $3105$ 
  unimodular triangulations up to symmetry. All triangulations are regular.
   For each such triangulation we computed the graph $G$ and the polyhedral cone
  $\mathbb{M}_\Delta$.  Each $\mathbb{M}_\Delta$ has dimension $3$, $4$, or $5$, with census given 
  on the left in Table~\ref{tab:moduli:hypg3}.
  For each cone $\mathbb{M}_\Delta$ we then checked that the inequalities stated in
Theorem~\ref{thm:hypg3} are satisfied. This proves that the dense realizable part of $\mathbb{M}_{R_{4,2}}$ is contained
 in the polyhedral space described by our constraints.

For the converse direction, we construct a planar tropical realization of each metric graph
that satisfies our constraints.  For the graph {\bf (020)}, we consider eleven cases:
\begin{itemize}
\vspace{-0.1in}
\item[(i)] $y<v+2w$, $u\neq v$, $y\neq z$; \hfill  ($\dim=5$) 
\vspace{-0.14in} 
\item[(ii)] $y=v+2w$, $u\neq v$, $y\neq z$; \hfill  ($\dim=5$) 
\vspace{-0.14in} 
\item[(iii)] $(\, y<v+3w$, $v<u<2v $, $y\neq z\,)$ or $(\, y<v+2w$, $u\neq v $, $y<z<2y\,)$; \hfill  ($\dim=5$) 
\vspace{-0.14in}
\item[(iv)] $(\, y<v+3w$, $u=2v $, $y\neq z\,)$ \hbox{ or } $(\, y<v+2w$, $u\neq v $, $z=2y\,)$; \hfill  ($\dim=4$) 
\vspace{-0.14in}
\item[(v)]  $(\, y<v+3w$, $v<u<2v$, $y= z\,)$  \hbox{ or } $(\, y<v+4w$, $u=v$, $y<z<2z\,)$; \hfill  ($\dim=4$) 
\vspace{-0.16in}
\item[(vi)]  $(\, y<v+3w$, $u=2v$, $y= z\,)$  \hbox{ or } $(\, y<v+4w$, $u=v$, $z=2y\,)$; \hfill  ($\dim=3$) 
\vspace{-0.14in}
\item[(vii)] $y=v+3w$, $v<u<2v$, $y\neq z$; \hfill  ($\dim=4$) 
\vspace{-0.14in}
\item[(viii)] $y=v+3w$, $u=2v $, $y\neq z$; \hfill  ($\dim=3$) 
\vspace{-0.16in}
\item[(ix)]  $(\, y<v+4w$, $u=v$, $y\neq z\,)$  \hbox{ or } $(\, y<v+2w$, $y=z$, $u\neq v\,)$; \hfill  ($\dim=3$) 
\vspace{-0.16in}
\item[(x)]  $y<v+4w$, $u= v$, $y= z$; \hfill  ($\dim=3$) 
\vspace{-0.16in}
\item[(xi)]  $y=v+4w$, $u= v$, $y\neq z$. \hfill  ($\dim=3$)\,\,
\end{itemize}
The disjunction of (i),(ii),\ldots,(xi) is equivalent to (\ref{eq:logiclogic1}).  Triangulations giving all metric graphs satisfying each case are pictured in Figure \ref{figure:(020)3_hyp}.  Next to the first triangulation is a metric graph arising from it.

\begin{figure}[h]
\centering
\includegraphics[scale=0.55]{hyp_graph_example}  \quad\,
\includegraphics[scale=.85]{_020_3_hyp_i} \quad
\includegraphics[scale=.85]{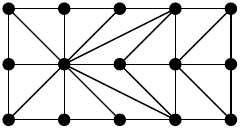} \quad
\includegraphics[scale=.85]{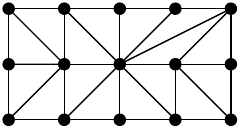} \quad
\includegraphics[scale=.85]{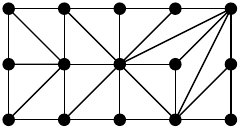} \quad
 \includegraphics[scale=.85]{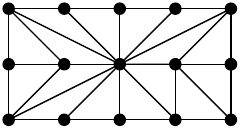} \quad
 \\ \vspace{0.14in} 
\includegraphics[scale=.85]{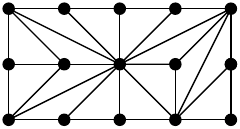} \quad
\includegraphics[scale=.85]{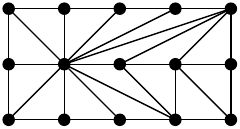} \quad
\includegraphics[scale=.85]{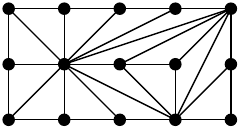} \quad
\includegraphics[scale=.85]{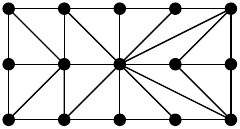} \quad
\includegraphics[scale=.85]{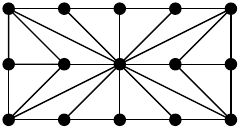} \quad
\includegraphics[scale=.85]{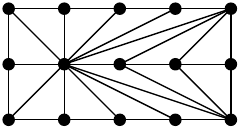}  \quad
\caption{Triangulations giving all realizable hyperelliptic metrics for the graph (020)}
\label{figure:(020)3_hyp}
\end{figure}

Next we deal with graph {\bf (111)}.  Here we
 need two triangulations, one for $u\neq v$ and one for $u=v$.  They are pictured in Figure \ref{figure:(111)3_hyp}.
The left gives $u\neq v$, and the middle gives $u=v$.

\begin{figure}[h]
\centering
\includegraphics[scale=1]{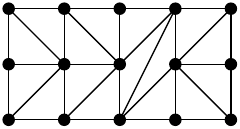} \qquad 
\includegraphics[scale=1]{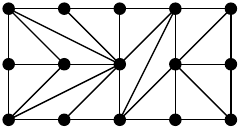} \qquad \qquad
\includegraphics[scale=1]{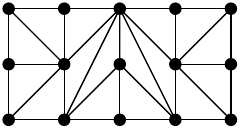} 
\caption{Triangulations realizing hyperelliptic metrics for the graphs (111) and (212)}
\label{figure:(111)3_hyp}
\end{figure}
Finally, for the graph {\bf (212)}, the single triangulation on the right in Figure \ref{figure:(111)3_hyp} suffices.
\end{proof}

\section{Genus Four}
\label{sec:genus4}

In this section we compute the moduli space  of tropical plane curves of genus $4$.
This is
$$ \mathbb{M}_4^{\rm planar} \,\,= \,\,
\mathbb{M}_{Q^{(4)}_1} \,\cup \, \mathbb{M}_{Q^{(4)}_2} \, \cup \, \mathbb{M}_{Q^{(4)}_3} \,\, \cup \,\, 
\mathbb{M}^{\rm planar}_{4,{\rm hyp}}, $$
where $Q^{(4)}_i$ are the three genus $4$ polygons in Proposition  \ref{prop:twelvepolygons}.
They are shown in Figure~\ref{figure:genus4_polygons}.
\begin{figure}[h]
\centering
\includegraphics[scale=1]{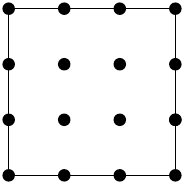}  \qquad \quad
\includegraphics[scale=1]{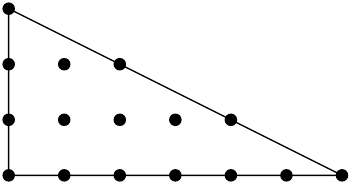}  \quad
\includegraphics[scale=1]{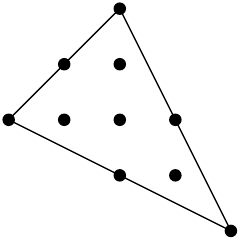}  \qquad
\includegraphics[scale=1]{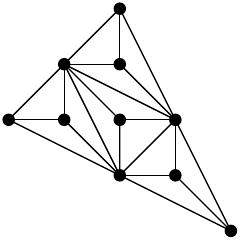} 
\caption{The three non-hyperelliptic genus $4$ polygons and a triangulation}
\label{figure:genus4_polygons}
\end{figure}

\begin{figure}[b]
  \centering
  \begin{subfigure}[b]{0.17\textwidth}
    \includegraphics[width=\textwidth]{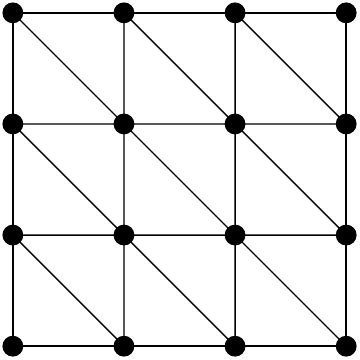}
    \caption*{(000)A}
  \end{subfigure}\,\,\,\,\,\,
  \begin{subfigure}[b]{0.17\textwidth}
    \includegraphics[width=\textwidth]{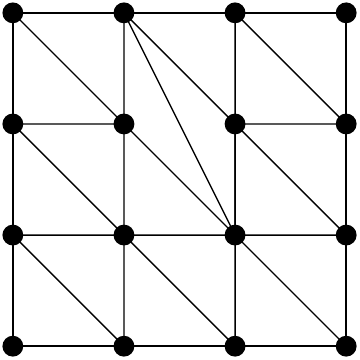}
    \caption*{(010)}
  \end{subfigure}\,\,\,\,\,\,
  \begin{subfigure}[b]{0.17\textwidth}
    \includegraphics[width=\textwidth]{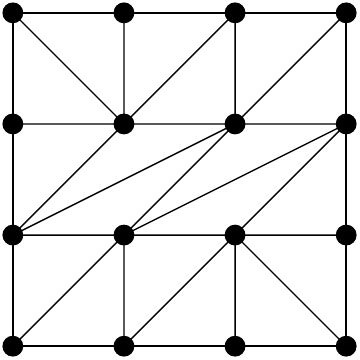}
    \caption*{(020)}
  \end{subfigure}\,\,\,\,\,\,
  \begin{subfigure}[b]{0.17\textwidth}
    \includegraphics[width=\textwidth]{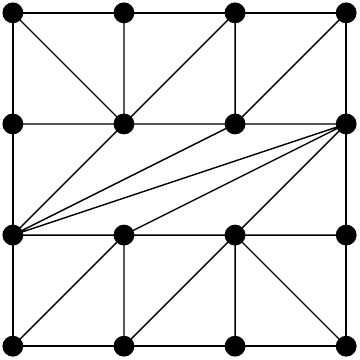}
    \caption*{(021)}
  \end{subfigure}\,\,\,\,\,\,
  \begin{subfigure}[b]{0.17\textwidth}
    \includegraphics[width=\textwidth]{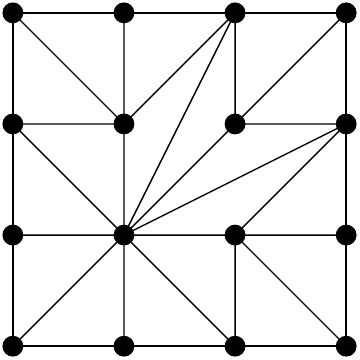}
    \caption*{(030)}
  \end{subfigure}
  \caption{Triangulations $\Delta$ of $Q^{(4)}_1$ with $\dim(\mathbb{M}_\Delta)=9$}
  \label{figure:genus4_9_dimensional}
\end{figure}

There are $17$ trivalent genus $4$ graphs, of which $16$ are planar.  These were first enumerated in
\cite{Ba}, and are shown in Figure \ref{figure:genus4_graphs}. 
  All have $6$ vertices and $9$ edges.  The labels $(\ell b c)$ are as in Section \ref{sec:genus3}: $\ell$ is the number of loops, $b$ the number of bi-edges, and $c$ the number of cut edges.  
This information is enough to uniquely determine the graph with the exception of (000),
where ``A'' indicates the honeycomb graph and ``B'' the complete bipartite graph~$K_{3,3}$.
\begin{figure}[h]
  \centering
  \begin{subfigure}[b]{0.15\textwidth}
    \includegraphics[width=\textwidth]{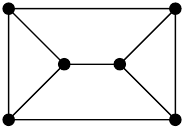}
    \caption*{(000)A}
  \end{subfigure} \quad
    \begin{subfigure}[b]{0.18\textwidth}
    \includegraphics[width=\textwidth]{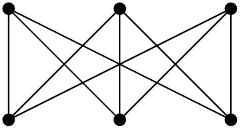}
    \caption*{(000)B}
  \end{subfigure} \quad \,\,
  \begin{subfigure}[b]{0.10\textwidth}
    \includegraphics[width=\textwidth]{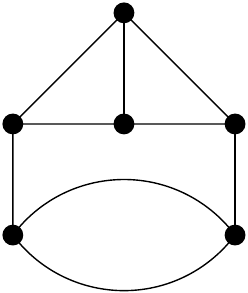}
    \caption*{(010)}
  \end{subfigure} \quad \,\,
  \begin{subfigure}[b]{0.15\textwidth}
    \includegraphics[width=\textwidth]{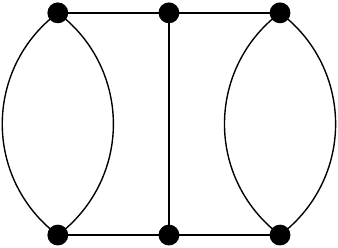}
    \caption*{(020)}
  \end{subfigure} \quad
  \begin{subfigure}[b]{0.18\textwidth}
    \includegraphics[width=\textwidth]{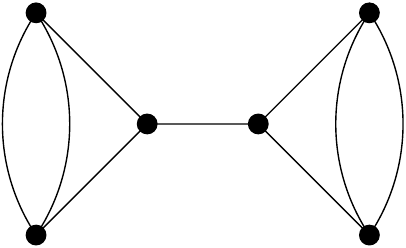}
    \caption*{(021)}
  \end{subfigure} \quad
  \begin{subfigure}[b]{0.18\textwidth}
    \includegraphics[width=\textwidth]{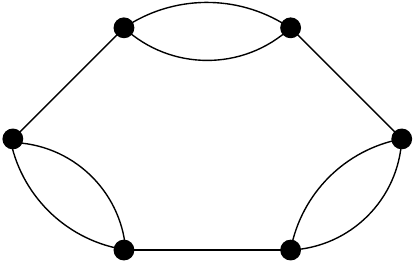}
    \caption*{(030)}
  \end{subfigure} \qquad
  \begin{subfigure}[b]{0.18\textwidth}
    \includegraphics[width=\textwidth]{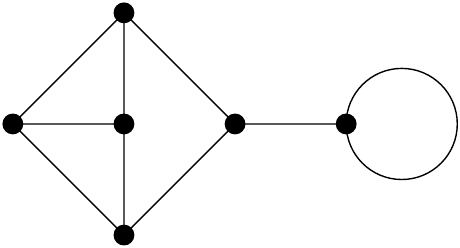}
    \caption*{(101)}
  \end{subfigure} \qquad
  \begin{subfigure}[b]{0.18\textwidth}
    \includegraphics[width=\textwidth]{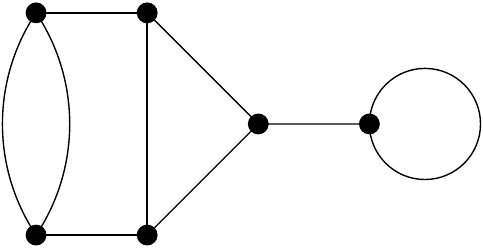}
    \caption*{(111)}
  \end{subfigure} \qquad
  \begin{subfigure}[b]{0.18\textwidth}
    \includegraphics[width=\textwidth]{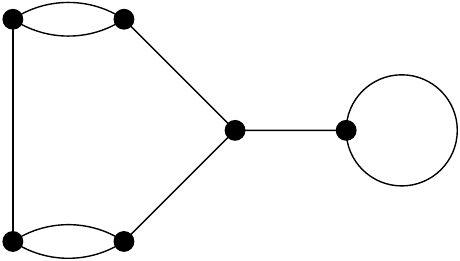}
    \caption*{(121)}
  \end{subfigure} \quad
  \begin{subfigure}[b]{0.18\textwidth}
    \includegraphics[width=\textwidth]{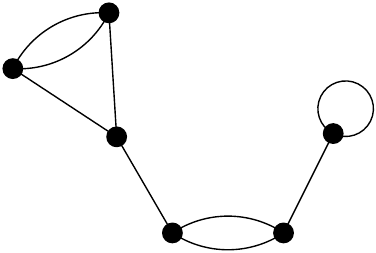}
    \caption*{(122)} 
  \end{subfigure} \qquad
  \begin{subfigure}[b]{0.15\textwidth}
    \includegraphics[width=\textwidth]{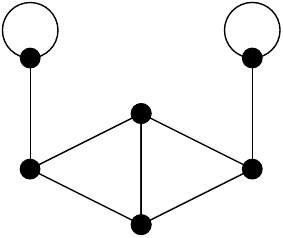}
    \caption*{(202)} 
  \end{subfigure} \qquad
  \begin{subfigure}[b]{0.15\textwidth}
    \includegraphics[width=\textwidth]{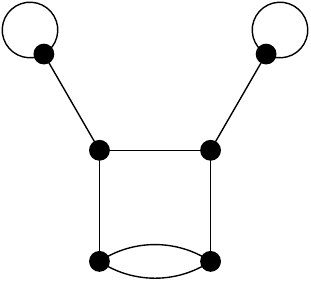}
    \caption*{(212)}
  \end{subfigure} \quad
  \begin{subfigure}[b]{0.18\textwidth}
    \includegraphics[width=\textwidth]{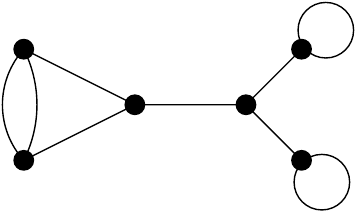}
    \caption*{(213)}
  \end{subfigure} \quad
  \begin{subfigure}[b]{0.18\textwidth}
    \includegraphics[width=\textwidth]{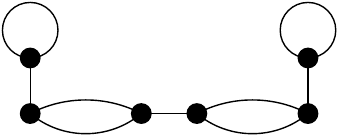}
    \caption*{(223)} 
  \end{subfigure} \qquad
  \begin{subfigure}[b]{0.18\textwidth}
    \includegraphics[width=\textwidth]{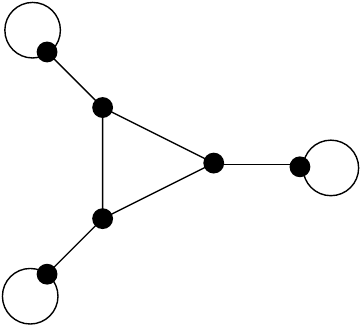}
    \caption*{(303)}
  \end{subfigure} \qquad
  \begin{subfigure}[b]{0.22\textwidth}
    \includegraphics[width=\textwidth]{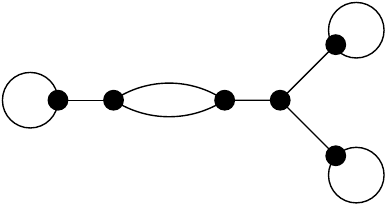}
    \caption*{(314)}
  \end{subfigure} \qquad
  \begin{subfigure}[b]{0.18\textwidth}
    \includegraphics[width=\textwidth]{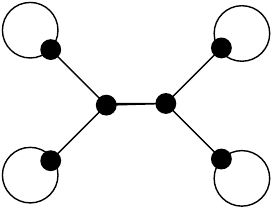}
    \caption*{(405)}
  \end{subfigure}
  \caption{The $17$ trivalent graphs of genus 4. All are planar except for (000)B.}
  \label{figure:genus4_graphs}
\end{figure}

Up to their respective symmetries,
 the square $Q^{(4)}_1 = R_{3,3}$ has $5941$ unimodular triangulations,
 the triangle $Q^{(4)}_2$ has $ 1278 $ unimodular triangulations,
 and the triangle $Q^{(4)}_3$ has $20$ unimodular triangulations.
 We computed the cone $\mathbb{M}_\Delta$ for each
 triangulation $\Delta$, and we ran the pipeline of
 Section \ref{sec:combinat}.
  We summarize our findings as the main result of this section:
 
\begin{theorem} \label{thm:genus4}
 Of the $17$ trivalent graphs, precisely $13$ are realizable by
 tropical plane curves.  The moduli space 
 $\,\mathbb{M}_4^{\rm planar}\,$
  is $9$-dimensional,
 but it is not pure: 
 the left decomposition in {\rm (\ref{eq:MPdecompose})}
  has components 
 {\rm (\ref{eq:MPGdecompose})} of dimensions $7$, $8$ and $9$.
 That decomposition is explained in   Table~\ref{genus_4_table}.
   \end{theorem}

The four non-realizable graphs are  (000)B,
(213), (314) and (405). This is obvious for (000)B,
because $K_{3,3}$ is not planar. The other three
are similar to the genus $3$ graph (303),
and are ruled out by  Proposition \ref{prop:sprawling} below.
The $13$ realizable graphs $G$
appear in the rows in Table \ref{genus_4_table}.
The first three columns correspond to the polygons
$Q^{(4)}_1$, $Q^{(4)}_2$ and $Q^{(4)}_3$. Each entry is
the number of regular unimodular triangulations 
$\Delta$ of $Q^{(4)}_i$ with skeleton $G$.
The entry is blank if no such triangulation exists.
Six of the graphs are realized by all three polygons,
five are realized by two polygons, and two are realized by 
only one polygon. For instance, the graph (303) comes from
a unique triangulation of the triangle $Q^{(4)}_3$,
shown on the right in Figure~\ref{figure:genus4_polygons}.
Neither $Q^{(4)}_1$ nor $Q^{(4)}_2$ can realize this graph.

Our moduli space $\mathbb{M}^{\rm planar}_4$
has dimension $9$. We know this already from 
Proposition~\ref{prop:trianglerectangle}, where
 the square $Q^{(4)}_1$ appeared as $ R_{3,3}$.
In classical algebraic geometry, that square serves as
the Newton polygon for canonical curves of  genus $4$ lying on a smooth quadric surface. In Table \ref{genus_4_table},
we see that all realizable graphs except for (303) arise from triangulations of $R_{3,3}$.
However, only five graphs allow for the maximal degree of freedom.
Corresponding triangulations are depicted in Figure \ref{figure:genus4_9_dimensional}.

 \begin{table}[t]
  \caption{
  The number of triangulations for the graphs of genus $4$
  and their moduli dimensions
    \label{genus_4_table}}
    \vspace{-0.14in}
  \centering
  \begin{tabular*}{.95\linewidth}{@{\extracolsep{\fill}}lrrrrrr@{}}
    \toprule
    $G$ &$\#\Delta_{Q^{(4)}_1,G}$ &$\#\Delta_{Q^{(4)}_2,G}$ &$\#\Delta_{Q^{(4)}_3,G}$ & 
    $\dim(\mathbb{M}_{Q^{(4)}_1,G})$ &$\dim(\mathbb{M}_{Q^{(4)}_2,G})$ & $\dim(\mathbb{M}_{Q^{(4)}_3,G})$ \\
    \midrule
        (000)A &$1823$&$127$&$12$& $9$ & $8$ &{$7$}\\
        (010) & $2192$ &$329$&$2$& $9$ & $8$ &$7$ \\
    (020) & $351$ &$194$&& $9$ & $8$ & \\
    (021) & $351$ &$3$&& $9$ & $7$ & \\
    (030) & $334$&$23$&$1$&$9$ & $8$ &$7$\\
    (101) & $440$ &$299$&$2$&$8$ & $8$ &$ 7$\\
    (111) &$130$&$221$&&$8$ & $8$ & \\
    (121) &$130$&$40$&$1$&$8$ & $8$ &$7$\\
    (122) & $130$ &$11$&&$8$ & $7$ &  \\
    (202) &$15$&$25$&&$7$ &$7$ & \\
    (212) &$30$&$6$&$1$&$7$ & $7$ &$7$\\
    (223) &$15$&&&$7$ & & \\
    (303) &&&1& & & $7$     \\
    \midrule
    total & 5941 & 1278 & 20 \\
    \bottomrule
  \end{tabular*}
\end{table}

 The last three columns in Table \ref{genus_4_table}
list the dimensions of the
moduli space $\mathbb{M}_{Q^{(4)}_i,G}$, which is
the maximal dimension of any cone $\mathbb{M}_\Delta$
where $\Delta$ triangulates $Q^{(4)}_i$ and has skeleton $G$.
More detailed information is furnished in Table  \ref{tab:moduli:g4R}.
The three subtables
(one each for $i=1,2,3$) explain the decomposition  (\ref{eq:MPGdecompose}) 
of each stacky fan $\mathbb{M}_{Q^{(4)}_i,G}$.  The row sums in
Table  \ref{tab:moduli:g4R} are the first three columns in 
Table \ref{genus_4_table}. For instance, the graph (030)
arises in precisely $23$ of the $1278$ triangulations $\Delta$ of the triangle $Q^{(4)}_2$.
 Among the corresponding cones
$\mathbb{M}_\Delta$, three have  dimension six, twelve have dimension seven,
and eight have dimension eight.
  
\begin{table}
  \caption{All cones $\mathbb{M}_\Delta$ from triangulations $\Delta$ of the three genus $4$ polygons in Figure~\ref{figure:genus4_polygons}}
  \label{tab:moduli:g4R}
  \centering
  \begin{tabular*}{\linewidth}{@{\extracolsep{\fill}}l@{\hskip .5in}rrrrr@{\hskip .5in}rrrr@{\hskip .5in}rrrr@{}}
    \toprule
    & \multicolumn{5}{c}{$Q^{(4)}_1$} & \multicolumn{4}{c}{$Q^{(4)}_2$} & \multicolumn{4}{c}{$Q^{(4)}_3$}  
    \\
    \midrule
    $G \backslash \!\! $ dim  & 5 & 6 & 7 & 8 & 9 & 5 & 6 & 7 & 8 &4 & 5 & 6 & 7 \\
    \midrule
    (000) & 103 & 480 & 764 & 400 &  76 &     5 &  52 &  60 &  10  & 1 & 6 & 3 & 2   \\
    (010) &  38 & 423 & 951 & 652 & 128 &    7 & 113 & 155 &  54  &   &   & 1 & 1 \\
    (020) &   3 &  32 & 152 & 128 &  36  &       &  53 & 100 &  41 &       &    &   & \\
    (021) &   3 &  32 & 152 & 128 &  36  &       &   1 &   2 &&       &    &   &  \\
    (030) &     &  45 & 131 & 122 &  36  &       &   3 &  12 &   8  &   &   &   & 1    \\
    (101) &  15 & 155 & 210 &  60 & &    19 & 122 & 128 &  30  &   &   & 1 & 1    \\
    (111) &     &  10 &  80 &  40 & &       &  52 & 126 &  43  &       &    &   &  \\
    (121) &     &  35 &  65 &  30 &  &      &   8 &  20 &  12 &   &   &   & 1    \\
    (122) &     &  10 &  80 &  40 &    &   &   &   & 1 &       &    &   &  \\
    (202) &     &     &  15 &     &   &       &     &  25 & &       &    &   &    \\
    (212) &     &  15 &  15 &     &  &       &   4 &   2 &  &   &   &   & 1  \\
    (223) &     &     &  15 &     &    &       &    &   & &       &    &   &  \\
   (303) &     &     &   &     &    &       &    &   &   &   &   &   & 1\\
   \bottomrule
  \end{tabular*}
\end{table}

Equipped with these data, we can now extend the probabilistic analysis of Corollary \ref{cor:g3:probability} from genus
$3$ to genus $4$.  As before, we assume that all $17$ trivalent graphs are equally likely and we fix the uniform
distribution on each $8$-simplex that corresponds to one of the $17$ maximal cones in the $9$-dimensional moduli space
$\mathbb{M}_4$.  The five graphs that occur with positive probability are those with $\dim(\mathbb{M}_{Q^{(4)}_1,G}) =
9$. Full-dimensional realizations were seen in Figure \ref{figure:genus4_9_dimensional}.  The result of our volume
computations is the following table:
\begin{center}
  \begin{tabular*}{.74\linewidth}{@{\extracolsep{\fill}}lrrrrr@{}}
    \toprule
    Graph       & (000)A & (010) &  (020) & (021) & (030) \\
    Probability & 0.0101 & 0.0129 &  0.0084 & 0.0164 & 0.0336 \\
        \bottomrule
  \end{tabular*}
\end{center}

In contrast to the exact computation in Corollary \ref{cor:g3:probability}, our probability
computations for genus $4$ rely on a Monte-Carlo simulation, with one million random samples for each
graph.

\begin{corollary}\label{cor:g4:probability}
  Less than $0.5$\% of all metric graphs of genus $4$ come from plane tropical curves. More precisely, the fraction is
  approximately $\,\vol(\mathbb{M}^{\rm planar}_4)/ \vol(\mathbb{M}_4) = 0.004788 $.
\end{corollary}


By Theorem \ref{thm:chains}, $\,\mathbb{M}^{\rm planar}_{4,{\rm hyp}} = \mathbb{M}_{E^{(g)}_{g+2}}$.  This space is
$7$-dimensional, with $6$ maximal cones corresponding to the chains (020), (021), (111), (122), (202), and (223).  The
graphs (213), (314), and (405) are hyperelliptic if given the right metric, but beyond not being chain graphs, these are
not realizable in the plane even as combinatorial types by Proposition \ref{prop:sprawling}.

\section{Genus Five and Beyond}
\label{sec:fivesix}

The combinatorial complexity of
trivalent graphs and  of regular triangulations
increases dramatically with $g$,
and one has to be judicious in deciding
what questions to ask and what computations
to attempt. One way to start is to rule out
families of trivalent graphs $G$ that cannot
possibly contribute to $\mathbb{M}^{\rm planar}_g$.
Clearly, non-planar graphs $G$ are
ruled out. We begin this section by identifying another excluded
class. Afterwards we examine our moduli space for
$g=5$, and we check which graphs arise from
the polygons $Q^{(5)}_i$ in Proposition  \ref{prop:twelvepolygons}.
 
\begin{definition}  A connected, trivalent, leafless graph $G$ is called \emph{sprawling}
 if there exists a vertex $s$ of $G$ such that $G \backslash \{s\}$ consists of three distinct components.
\end{definition}

\begin{remark}
\rm  Each component of  $G\backslash \{s\}$ must have genus at least one; otherwise $G$ would not have been leafless. The vertex $s$ need not be unique.  The genus $3$ graph (303) in Figure~\ref{figure:genus3_graphs} is sprawling, as are the genus $4$ graphs (213), (314), and (405) in Figure \ref{figure:genus4_graphs}.
\end{remark}

\begin{proposition}  \label{prop:sprawling}
Sprawling graphs are never the skeletons of smooth tropical plane curves.
\end{proposition}
This was originally proven in \cite[Prop.~4.1]{CDMY}.  We present our own proof for completeness.
 
\begin{proof}  
Suppose the skeleton of a smooth tropical plane curve $C$ is a sprawling graph $G$ with separating vertex $s$. After a change of coordinates, we may assume that the directions emanating from $s$ are $(1,1)$, $(0,-1)$, and $(-1,0)$. The curve $C$ is dual to a unimodular triangulation $\Delta$ of a polygon $P\subset \R^2$. Let $T\in \Delta$ be the triangle dual to $s$.  We may take $T=\text{conv}\{(0,0),(0,1),(1,0)\}$ after an appropriate translation of $P$.
Let $P_1,P_2,P_3$ be the subpolygons of $P$ corresponding to the components of
 $G \backslash \{s\}$.  After relabeling, we have 
 $P_1\cap P_2=\{(0,1)\}$, $P_1\cap P_3=\{(0,0)\}$, and $P_2\cap P_3=\{(1,0)\}$.  Each $P_i$
 has at least one interior lattice point, since each component of $G\backslash\{s\}$ must have genus at least $1$.

\begin{figure}[h]
\centering
\includegraphics[scale=0.92]{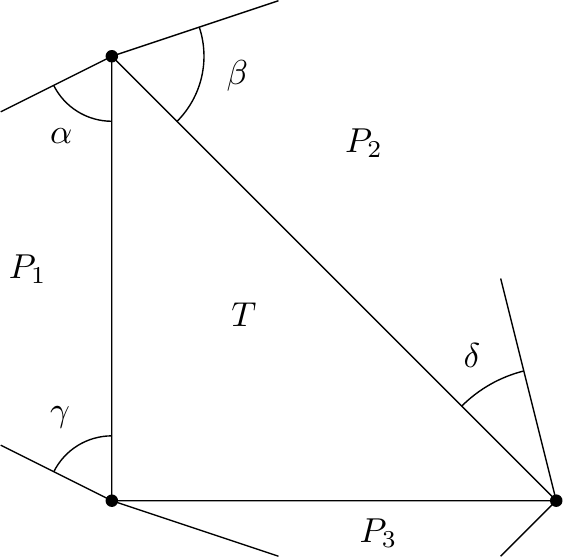}
\caption{The triangle $T$ with angles formed between it and the boundary edges of $P$}
\label{figure:sprawling_triangle}
\end{figure}

Let $\alpha,\beta,\gamma,\delta$ be the angles between the triangle $T$ and the boundary edges of $P$ emanating from the vertices of $T$, as pictured in Figure \ref{figure:sprawling_triangle}.  Since $P$ is convex, we know $\alpha+\beta\leq 3\pi/4$, $\gamma<\pi/2$, and $\delta <3\pi/4$.  As $P_1$ contains at least one interior lattice point, and $\gamma<\pi/2$, we must also have that $\alpha>\pi/2$;  otherwise $P_1\subset (\infty,0]\times[0,1]$, which has no interior lattice points.  Similarly, as $P_2$ has at least one interior lattice point and $\delta <3\pi/4$, we must have $\beta>\pi/4$.  But we now have that $\alpha+\beta>\pi/2+\pi/4=3\pi/4$, a contradiction.  Thus, the skeleton of  $C$ cannot be a sprawling graph, as originally assumed.
\end{proof}

\begin{remark}  \rm If $G$ is sprawling then $\mathbb{M}^{\rm planar}_g\cap \mathbb{M}_G \not=\emptyset$
because edge lengths can become zero
 on the boundary. However, it is only in taking  closures of spaces of realizable graphs that this intersection becomes nonempty.
\end{remark}

We will now consider the moduli space of tropical plane curves of genus $5$.
That space is
$$ \mathbb{M}_5^{\rm planar} \,\,= \,\,
\mathbb{M}_{Q^{(5)}_1} \,\cup \, \mathbb{M}_{Q^{(5)}_2} \, \cup \, 
\mathbb{M}_{Q^{(5)}_3}  \, \cup \, \mathbb{M}_{Q^{(5)}_4} \,\, \cup \,\, 
\mathbb{M}^{\rm planar}_{5,{\rm hyp}}, $$
where  $Q^{(5)}_1$, $Q^{(5)}_2$, $Q^{(5)}_3$, $Q^{(5)}_4$ 
are the four genus $5$ polygons in Proposition  \ref{prop:twelvepolygons}.
They are shown in Figure~\ref{figure:genus5_polygons}. Modulo their
respective symmetries,
the numbers of unimodular triangulations of these polygons are:  
$508$ for $Q^{(5)}_1$, $147908$ for $Q^{(5)}_2$, $162$ for $Q^{(5)}_3$, and $968$ for $Q^{(5)}_4$.


\begin{figure}[h]
\centering
\includegraphics[scale=1.1]{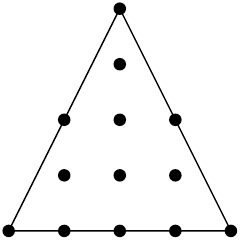}  \qquad \quad
\includegraphics[scale=1.1]{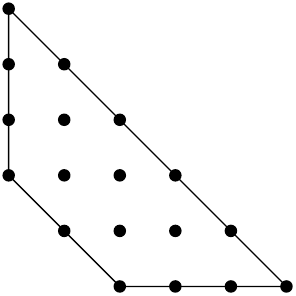}  \quad
\includegraphics[scale=1.1]{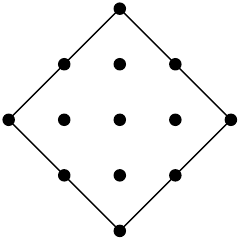}  \qquad
\includegraphics[scale=1.1]{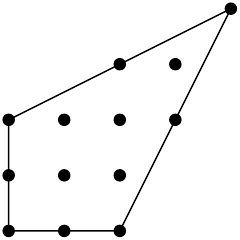} 
\caption{The genus $5$ polygons $Q^{(5)}_1$, 
 $Q^{(5)}_2$,  $Q^{(5)}_3$ and  $Q^{(5)}_4$}
\label{figure:genus5_polygons}
\end{figure}

We applied the pipeline described in Section \ref{sec:combinat}
to all these triangulations. The outcome of our computations is the following result
which is the genus $5$ analogue to
Theorem~\ref{thm:genus4}.  

\begin{theorem} \label{thm:genus5}
 Of the $71$ trivalent graphs of genus $5$, precisely $38$ are realizable by smooth
 tropical plane curves. The  four polygons
satisfy $\dim(\mathbb{M}_{Q^{(5)}_i})=9,11,10,10$ for $i=1,2,3,4$.
    \end{theorem}

 All but one of the $38$ realizable graphs arise from $Q^{(5)}_1$ or $Q^{(5)}_2$.  The remaining graph, realized only by a single triangulation of $Q^{(5)}_4$, is illustrated in Figure \ref{figure:genus5_unique}.  This is reminiscent of the genus $4$ graph (303), which was realized only by the triangulation of $Q_3^{(4)}$ in Figure~\ref{figure:genus4_polygons}.   The other $37$ graphs are realized by at least two of the polygons  $Q^{(5)}_1, \ldots,  Q^{(5)}_4, E^{(5)}_7$.

 \begin{figure}[h]
\centering
\includegraphics[scale=1.4]{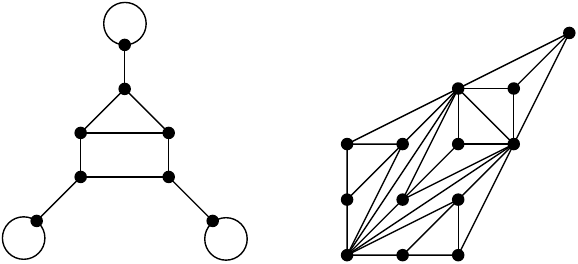} 
\caption{A genus $5$ graph, and the unique triangulation that realizes it}
\label{figure:genus5_unique}
\end{figure}
 
 Among the $71$ trivalent graphs of genus $5$,
 there are four non-planar graphs and $15$ sprawling graphs, with none both non-planar and 
 sprawling.  This left us with $52$ possible candidates for realizable graphs.  We ruled out the
 remaining $14$  by the explicit computations described in Section \ref{sec:combinat}.
 Three of these $14$ graphs are shown in Figure \ref{figure:genus5_exceptions}.
 At present we do not know any general rule that discriminates between
  realizable and non-realizable graphs.
 
 \begin{figure}[h]
\centering
\includegraphics[scale=1.25]{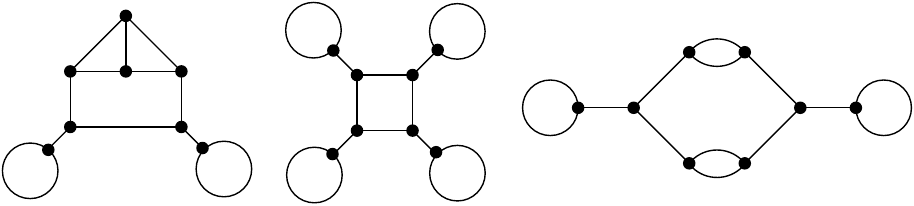} 
\caption{Some non-realizable graphs of genus 5}
\label{figure:genus5_exceptions}
\end{figure}

The process we have carried out for genus $g=3$, $4$, and $5$ can be continued for $g\geq 6$.  As the genus increases so
does computing time, so it may be prudent to limit the computations to special cases of interest.  For $g=6$ we might
focus on the triangle $Q^{(6)}_1=T_5$.  This is of particular interest as it is the Newton polygon of a smooth plane
quintic curve.  This triangle has 561885 regular unimodular triangulations up to symmetry.

Although $T_5$ is interesting as the Newton polygon of plane quintics, it has the downside that $\mathbb{M}_{T_5}$ is
not full-dimensional inside $\mathbb{M}_6^{\rm planar}$.  Proposition \ref{prop:trianglerectangle} implies that
$\dim(\mathbb{M}_{T_5})=12$, while $ \dim(\mathbb{M}_6^{\rm planar})=13$, and this dimension is attained by the
rectangle $R_{3,4}$ as in \eqref{eq:trapezoid1}.

This might lead us to focus on \emph{full-dimensional polygons} of genus $g$.  By this we mean polygons $P$ whose moduli
space $\mathbb{M}_P$ has the dimension in \eqref{eq:dimformula}.  For each genus from $3$ to~$5$, our results show that
there is a unique full-dimensional polygon, namely, $T_4$, $R_{3,3}$, and $Q_2^{(5)}$.  The proof of Theorem
\ref{thm:dimension} furnishes an explicit example for each genus $g\geq 6$: take the rectangle in \eqref{eq:trapezoid1}
or the trapezoid in \eqref{eq:trapezoid2} if $g\neq 7$, or the hexagon $H_{4,4,2,6}$ if $g=7$.  Calculations show that there are exactly two full-dimensional maximal polygons for $g=6$, namely, $Q_3^{(6)}=R_{3,4}$ and $Q_4^{(6)}$ from Proposition \ref{prop:twelvepolygons}.

We conclude with several open questions.

\begin{question}\rm{Let $P$ be a maximal lattice polygon with at least $2$ interior lattice points.
\begin{itemize}

\item[(1)]  What is the relationship between $\mathbb{M}_P$ and $\mathcal{M}_P$?  In particular, does the equality ${\rm
  dim}(\mathbb{M}_P) = {\rm dim}(\mathcal{M}_P)$ hold for all  $P$?

\item[(2)]  How many $P$ with $g$ interior lattice points give a full dimensional $\mathbb{M}_P$ inside $\mathbb{M}_g^{\rm planar}$?

\item[(3)]  Is there a more efficient way of determining if a combinatorial graph of genus $g$ appears in $\mathbb{M}_g^{\rm planar}$ than running the pipeline in Section \ref{sec:combinat}?

\end{itemize}
}
\end{question}

\bigskip
\bigskip

\noindent
{\bf Acknowledgements.}\\
We thank Wouter Castryck and John Voight for helpful comments on a draft of this paper.  Sarah Brodsky was supported by
a European Research Council grant SHPEF awarded to Olga Holtz.  Michael Joswig received support from the Einstein
Foundation Berlin and the Deutsche Forschungsgemeinschaft.  Ralph Morrison and Bernd Sturmfels were 
supported by the US National Science Foundation.

\begin{small}

\bigskip

\end{small}

\footnotesize 

\noindent \textbf{Authors' addresses:}

\noindent
Sarah Brodsky, Technische Universit\"at Berlin,
10623 Berlin, Germany, \texttt{brodsky@math.tu-berlin.de}

\noindent
Michael Joswig, Technische Universit\"at Berlin, MA 6-2,
10623 Berlin, Germany,
\texttt{joswig@math.tu-berlin.de}

\noindent Ralph Morrison,  University of California, Berkeley, CA 94720-3840, USA,
\texttt{morrison@math.berkeley.edu}

\noindent Bernd Sturmfels,  University of California, Berkeley, CA 94720-3840, USA,
\texttt{bernd@math.berkeley.edu}

\end{document}